\documentclass{amsart}
\usepackage{amsmath,amsfonts,amssymb,amsthm,mathtools}
\usepackage{amscd}
\usepackage{bbm}
\usepackage{enumerate}
\usepackage{galois}
\usepackage{mathrsfs}
\usepackage{xypic}
\usepackage{geometry}
\usepackage{hyperref}

\usepackage{color}

\theoremstyle{plain}
\newtheorem{Th}{Theorem}[section]
\newtheorem{Cor}[Th]{Corollary}
\newtheorem{Lem}[Th]{Lemma}
\newtheorem{Prop}[Th]{Proposition}
\newtheorem{Conj}[Th]{Conjecture}

\newtheorem{Def}{Definition}[section]
\newtheorem{Ex}{Example}[section]
\newtheorem{Rem}{Remark}[section]

\numberwithin{equation}{section}

\DeclareMathOperator{\ima}{Im}
\DeclareMathOperator{\vbh}{\mathit{VBH}}
\DeclareMathOperator{\bh}{\mathit{BH}}

\newcommand{\der}{\mathrm{Der}(\hm A)}
\newcommand{\diff}[2]{\frac{\partial #1}{\partial #2}}

\newcommand{\qa}{\alpha}
\newcommand{\qb}{\beta}
\newcommand{\qd}{\delta}
\newcommand{\qg}{\gamma}
\newcommand{\qs}{\sigma}
\newcommand{\qt}{\tau}
\newcommand{\qth}{\theta}
\newcommand{\qe}{\varepsilon}

\newcommand{\qp}{\partial}
\newcommand{\qo}{\omega}

\newcommand{\Qg}{\Gamma}
\newcommand{\ql}{\lambda}

\newcommand{\Qd}{\Delta}
\newcommand{\Qo}{\Omega}

\newcommand{\vard}[2]{\frac{\delta #1}{\delta #2}}
\newcommand{\bm}[1]{\bar{\mathcal #1}}
\newcommand{\hm}[1]{\hat{\mathcal #1}}
\newcommand{\aij}[2]{\partial_{#1}f^{#2}}
\newcommand{\bij}[2]{f^{#1}\frac{\partial_{#1}f^{#2}}{f^{#2}}}
\newcommand{\xym}[1]{\begin{center}\leavevmode \xymatrix{#1} \end{center}}
\newcommand{\lrb}[1]{\left(#1\right)}
\newcommand{\kk}[1]{\left(#1\right)}
\newcommand{\ee}[1]{\begin{equation}#1\end{equation}}
\newcommand{\aaa}[1]{\begin{align}#1\end{align}}
\newcommand{\fk}[2]{\left[#1, #2\right]}

\begin{document}

\title{Variational Bihamiltonian Cohomologies and Integrable Hierarchies I: Foundations}

\author{Si-Qi Liu, Zhe Wang, Youjin Zhang}

\begin{abstract}
This series of papers is devoted to the study of deformations of  Virasoro symmetries of the principal hierarchies associated to  semisimple Frobenius manifolds. The main tool we use is a generalization of the bihamiltonian cohomology called the \textit{variational bihamiltonian cohomology}. In the present paper, we give its definitions and compute the associated  cohomology groups that will be used in our study of deformations of Virasoro symmetries. To illustrate its application, we classify the conformal bihamiltonian structures with semisimple hydrodynamic limits.
\end{abstract}

\date{\today}

\maketitle
\tableofcontents

%%%%%%%%%%%%%%%%%%%%%%%%%%%%%%%%%%%%%%%%%%%%%%%
%%%%%%%%%%%%%%%%%%%%%%%%%%%%%%%%%%%%%%%%%%%%%%%
%%%%%%%%%%%%%%%%%%%%%%%%%%%%%%%%%%%%%%%%%%%%%%%

\section{Introduction}\label{intro}
Since Dubrovin introduced the notion of Frobenius manifold in the beginning of 90's of the last century \cite{dubrovin1993integrable, dubrovin1996geometry} inspired by the development of 2d topological field theory \cite{dijkgraaf1991notes,dijkgraaf1991topological, kontsevich1992intersection,witten1990structure, witten1990two}, its relationship with hierarchies of integrable PDEs has become an important research subject in the theory of integrable systems and its applications in the study of Gromov-Witten invariants, singularity theory and quantum field theory, see some of the related works  
\cite{buryak2015double,
dubrovin1998bihamiltonian, 
dubrovin1999frobenius, dubrovin2001normal, dubrovin2004virasoro, 
eguchi1997quantum,
%%%%%%%%%%%%%%%%%%%%%%%%%%%
eguchi1995genus,fan2013witten,
getzler2001toda,givental2001gromov,  
givental2001semisimple, 
givental2005simple, 
liu2015bcfg, milanov2016gromov,zhang2002cp} and references therein.
In \cite{dubrovin2001normal}, Dubrovin and Zhang, the third author of the present paper, proposed a project to classify hierarchies of integrable PDEs which possess hydrodynamic limits and satisfy the following properties: each integrable hierarchy has a bihamiltonian structure, a tau function (also called a tau structure), and an infinite set of Virasoro symmetries that act linearly on the tau function. The bihamiltonian structure of the integrable hierarchy is assumed to be represented as a formal power series w.r.t. the dispersion parameter $\qe$, its coefficients are given by homogeneous differential polynomials of the dependent variables of the integrable hierarchy, and its dispersionless limit
is a bihamiltonian structure of hydrodynamic type. The existence of a bihamiltonian structure and of a tau function implies that the dispersionless limit of the integrable hierarchy can be described by a Frobenius manifold structure or a degenerate one. On the other hand, starting from any Frobenius manifold, one can construct an integrable hierarchy of hydrodynamic type which is called the Principal Hierarchy of the Frobenius manifold. The flat metric and the intersection form of the Frobenius manifold yield a bihamiltonian structure of the Principal Hierarchy which also possesses a tau function and an infinite set of Virasoro symmetries. The Virasoro symmetries, apart from the first two, of the Principal Hierarchy do not act linearly on the tau function. Under the semisimplicity condition of the Frobenius manifold, the requirement of linear actions of the Virasoro symmetries on the tau function leads to a dispersionful deformation, called the topological deformation, of the Principal hierarchy. As it is shown in \cite{dubrovin2001normal}, the topological deformation of the Principal Hierarchy is constructed via a quasi-Miura transformation determined by the loop equation of the semisimple Frobenius manifold.

To classify dispersionful deformations of the Principal Hierarchy, Dubrovin and Zhang introduced in \cite{dubrovin2001normal} the notion of bihamiltonian cohomology $\bh^k(M; P_0, P_1)$ for a bihamiltonian structure $(P_0, P_1)$ of hydrodynamic type defined on the jet space of a smooth manifold $M$. The cohomology groups $\bh^2(M; P_0, P_1)$ and $\bh^3(M; P_0, P_1)$ characterize the equivalence classes of infinitesimal deformations of $(P_0, P_1)$ under Miura type transformations and the obstruction of extending an infinitesimal deformations to a full deformation respectively. Here we emphasize that the deformed bihamiltonian structures are required to have the property that the coefficients of the powers of the deformation parameter are given by homogeneous differential polynomials of the dependent variables, this property is called the {\em polynomiality} of the deformed Hamiltonian structures. 

In \cite{lorenzoni2002deformations} Lorenzoni studied the deformations of the bihamiltonian structure of the dispersionless KdV hierarchy and showed,
up to the fourth order approximation of the deformation parameter, that they are parametrized by a function of one variable.
The very first computation of the bihamiltonian cohomology groups for a semisimple bihamiltonian structure of hydrodynamic type is given in \cite{ DLZ-1,liu2005deformations}, from which the cohomology group $BH^2_{\ge 2}(M; P_0, P_1)$ can be obtained, and it shows that the equivalence classes of the infinitesimal deformations of $(P_0, P_1)$ with $n$ dependent variables are parametrized by a set of smooth functions $c_1(u_1), \dots, c_n(u_n)$, where $u_1, \dots, u_n$ are the canonical coordinates of the semisimple bihamiltonian structure of hydrodynamic type. These functions are called the central invariants of the deformed bihamiltonian structure. It is also conjectured that the cohomology group $\bh^3(M; P_0, P_1)$ is trivial, and that any infinitesimal deformation of $(P_0, P_1)$ can be extended to a full deformation. 

An important observation toward computing the cohomology group $\bh^3(M; P_1, P_2)$ was made by the first and third author of the present paper in \cite{liu2013bihamiltonian}, where they showed that one can work in the space of differential polynomials instead of the space of the local functionals to compute the bihamiltonian cohomology. By using this idea they proved the afore mentioned conjecture for the bihamiltonian structure of hydrodynamic type associated to the one dimensional Frobenius manifold (or the dispersionless KdV hierarchy). Base on the method of \cite{liu2013bihamiltonian} and on a clever utilizing of the technique of spectral sequences and some other tools from homological algebra, Carlet, Posthuma and Shadrin proved this conjecture for a general semisimple bihamiltonian structure of hydrodynamic type in \cite{carlet2018deformations}. 
 
A deformation of the bihamiltonian structure of the Principal hierarchy of a semisimple Frobenius manifold yields a deformation of the integrable hierarchy. It is shown in \cite{dubrovin2018bihamiltonian} that the deformed integrable hierarchy possesses a tau structure when the central invariants of the deformed bihamiltonian structure are constant functions. 
However, it remains to be unclear whether this deformed integrable hierarchy also possesses an infinite set of Virasoro symmetries. 
In particular, it is conjectured that the deformation of the Principal Hierarchy with all the central invariants being equal to $\frac1{24}$ possesses an infinite set of Virasoro symmetries which act linearly on the tau function of the deformed integrable hierarchy, or in other words, it coincides with the topological deformation of the Principal Hierarchy. 
An approach to prove the validity of this conjecture is as follows: one needs to show that the quasi-Miura transformation, which transforms the Principal Hierarchy to its topological deformation, also transforms the bihamiltonian structure of hydrodynamic type to a deformed bihamiltonian structure which satisfies the polynomiality property. In \cite{buryak2012deformations, buryak2012polynomial} Buryak, Posthuma and Shadrin proved the polynomiality property of the first deformed Hamiltonian structure of the Principal Hierarchy, and in the recent preprint 
\cite{iglesias2021bi} Iglesias and Shadrin provide a certain evidence of the validity of the polynomiality property of the second Hamiltonian structure.

The main purpose of the present paper and of its subsequents is to show that the deformation of the Principal Hierarchy of a semisimple Frobenius manifold which is associated to a deformed bihamiltonian structure with constant central invariants possesses an infinite set of Virasoro symmetries, and these Virasoro symmetries can be lifted to the actions on the tau function of the deformed integrable hierarchy; moreover, for the case when all the central invariants equal to $\frac1{24}$, the Virasoro symmetries act linearly on the tau function, and so this deformed integrable hierarchy is equivalent to the topological deformation of the Principal Hierarchy. In particular, this implies that the topological deformation of the Principal Hierarchy of a semisimple Frobenius manifold possesses a bihamiltonian structure which satisfies the polynomiality property. To this end, we first generalize in the present paper the notion of bihamiltonian cohomology by considering a certain cohomology, called the variational bihamiltonian cohomology, on the space of variational $1$-forms of the infinite jet space of the dependent variables, and apply it to the study of Virasoro symmetries of the deformed integrable hierarchy in the subsequent papers.

The paper is organized as follows. In Sect.\,\ref{ms}, we recall the basic notions of bihamiltonian structures in terms of local functionals on the infinite jet bundle of a super manifold, and we present the main results of this paper. In Sect.\,\ref{vbc}, we first define the variational Hamiltonian cohomology and prove the triviality of the cohomology groups, and then we define the variational bihamiltonian cohomology. Sect.\,\ref{vbh23} and Sect.\,\ref{vsh} are both devoted to the computation of the variational bihamiltonian cohomology groups. As an illustration of the application of this theory, we classify the conformal bihamiltonian structures and their deformations in Sect.\,\ref{conformal}. In Sect.\,\ref{con} we give some concluding remarks.

\section{Basic notions and Main results}\label{ms}
In this section, we first recall some basic definitions and constructions on bihamiltonian structures in terms of the language of super variables. The readers may refer to \cite{liu2011jacobi, liu2013bihamiltonian} and the  lecture note \cite{liu2018lecture} for a detailed introduction to this topic. 

Let $M$ be an $n$-dimensional smooth manifold, and $\hat M = \Pi(T^*M)$ be the super manifold of dimension $(n\mid n)$ obtained by reversing the parity of the fiber of the cotangent bundle of $M$. Locally, one may choose a coordinate system $(u^1,\cdots,u^n)$ for an open neighborhood $U\subset M$ and we will use  $(\qth_1,\cdots,\qth_n)$ to denote the coordinate system of the fibers. The variables $\qth_i$ are fermion, meaning that they satisfy the anti-commutation relations $\qth_i\qth_j + \qth_j\qth_i = 0$. Let $J^\infty(\hat M)$ be the infinite jet bundle of $\hat M$ and let $\hat{\mathcal A}$ be the ring of differential polynomials. In local coordinates, we can write it as
\[\hat{\mathcal A} = C^\infty(U)[[u^{i,s+1}, \theta^s_i\mid i = 1,\cdots,n;s\geq 0]].\]
Here and henceforth we denote $u^{i,0} = u^i$ and ${\qth_i^0} = \qth_i$. The jet variables $u^{i,s}, \qth_i^s$ for $s\geq 1	$ may be regarded as $s$-th order derivatives with respect to an independent variable, which we usually denote as $x$,  when regarding $u^i$ and $\qth_i$ as dependent variables. This identification is possible due to the existence of the global vector field
\begin{equation}\label{dx-zh}
\qp_x = \sum_{s\geq 0}u^{i,s+1}\diff{}{u^{i,s}}+\qth_i^{s+1}\diff{}{\qth_i^s}.
\end{equation}
Sometimes we will use a prime to denote $\qp_x$, i.e. $f':=\qp_xf$. The elements in the quotient space $\hat{\mathcal F} = \hat{\mathcal A}/\qp_x$ are called local functionals and for an element $f\in\hat{\mathcal A}$, we will use $\int f$ to denote its class in $\hat{\mathcal F}$. Let us define two degrees $\deg_{x}$ and $\deg_\qth$ on the space $\hm A$, which are called respectively the differential gradation and the super gradation. On the generators of $\hm A$, these degrees are defined by
\[
\deg_x u^{i,s} = \deg_x \qth^{s}_i = s;\quad \deg_\qth u^{i,s} = 0,\quad \deg_\qth \qth_i^s = 1.
\]
We denote the space of homogeneous elements by
\[
\hat{\mathcal A}_d = \{f\in\hat{\mathcal A}\mid \deg_x f = d\},\quad \hat{\mathcal A}^p = \{f\in\hat{\mathcal A}\mid  \deg_\qth f = p\},\quad \hat{\mathcal A}^p_d = \hat{\mathcal A}^p\cap \hat{\mathcal A}_d.
\]
It is clear that the vector field $\qp_x$ is homogeneous with respect to both degrees, hence the gradations on $\hat{\mathcal A}$ induce a gradation on $\hat{\mathcal F}$ and we will denote the space of homogeneous elements by $\hat{\mathcal F}_d$, $\hat{\mathcal F}^p$ and $\hat{\mathcal F}^p_d$ accordingly.

A nontrivial construction called the Schouten-Nijenhuis bracket equips the space $\hat{\mathcal F}$ with a graded Lie algebra structure. Define the variational derivative of an element in $\hat{\mathcal A}$ by
\[
\vard{f}{u^i} = \sum_{s\geq 0}(-\qp_x)^s\diff{f}{u^{i,s}},\quad \vard{f}{\qth_i^s} = \sum_{s\geq 0}(-\qp_x)^s\diff{f}{\qth_i^s}.
\]
It can be shown that the variational derivative annihilates the image of $\qp_x$ so we can define the variational derivative of a local functional as follows:
\[\vard{F}{u^i} = \int \vard{f}{u^i},\quad \vard{F}{\qth_i} = \int \vard{f}{\qth_i},\quad F = \int f\in\hat{\mathcal F}.
\]
Then the Schouten-Nijenhuis bracket is defined to be the following bilinear map 
\[[-,\ -]: \hat{\mathcal F}^p\times \hat{\mathcal F}^q\to\hat{\mathcal F}^{p+q-1},\quad
[P,Q] = \int \sum_i \vard{P}{\qth_i}\vard{Q}{u^i}+(-1)^p\vard{P}{u^i}\vard{Q}{\qth_i}.
\]
This bracket defines a graded Lie algebra structure on $\hm F$, whose sign convention is different from the usual definition, for more details one may refer to \cite{liu2011jacobi}.

For a local functional $P\in\hat{\mathcal F}^p$, we can associate a derivation on $\hat{\mathcal A}$ defined by
\begin{equation}
\label{def-d}
D_P = \sum_{i}\sum_{s\geq 0}\qp_x^s\left(\vard{P}{\qth_i}\right)\diff{}{u^{i,s}}+(-1)^p\qp_x^s\left(\vard{P}{u^i}\right)\diff{}{\qth^{s}_i}.
\end{equation}
We see that $D_P$ is a derivation of super degree $p-1$ and that $[D_P,\qp_x] = 0$. The following property is important for the construction of the cohomologies that will be given later in this and the next section:
\begin{equation}
\label{comm-d}
(-1)^{p-1}D_{[P,Q]} = D_P\comp D_Q-(-1)^{(p-1)(q-1)}D_Q\comp D_P;\quad P\in\hat{\mathcal F}^p,\ Q\in\hat{\mathcal F}^q.
\end{equation}
This relation shows that the map $D$ induces a (graded) Lie algebra homomorphism
\[\hm F\to\mathrm{Der}(\hat{\mathcal A}):\quad P\mapsto D_P,\]
here the natural graded Lie algebra structure on the space of derivations on $\hat{\mathcal A}$, denoted by $\mathrm{Der}(\hat{\mathcal A})$, is given by the graded commutators.
In what follows, we will also call an element of $\mathrm{Der}(\hat{\mathcal A})$ a vector field on $J^\infty(\hat M)$.

Using the above notions, a Hamiltonian structure is defined to be a local functional $P\in\hat{\mathcal F}^2$ satisfying $[P,P] = 0$, and it is called of hydrodynamic type if $P\in \hat{\mathcal F}^2_1$. A bihamiltonian structure is a pair of Hamiltonian structure $(P_0, P_1)$ satisfying an additional compatibility condition $[P_0, P_1] = 0$. Now assume that we have a bihamiltonian structure $(P_0,P_1)$ of hydrodynamic type \cite{liu2018lecture}, locally we represent them as follows:
\[
P_a = \frac 12\int\sum g^{ij}_a(u)\qth_i\qth_j^1+\Qg^{ij}_{a,k}(u)u^{k,1}\qth_i\qth_j,\quad a = 1,2.
\]
This bihamiltonian structure is called semisimple if the roots $\ql^1(u),\cdots,\ql^n(u)$ of the characteristic equation $\det(g_1^{ij}-\ql g_0^{ij}) = 0$ are distinct and not constant.

It is proved in \cite{ferapontov2001compatible} that if $(P_0, P_1)$ is semisimple, then the roots $\ql^1(u),\cdots,\ql^n(u)$ can serve as local coordinates of the manifold $M$ and in such a coordinate system $(P_0, P_1)$ can be represented in the following forms:
\[
P_0 = \frac 12\int \sum f^i(\ql)\qth_i\qth_i^1+ A^{ij}\qth_i\qth_j,\quad P_1 = \frac 12\int \sum g^i(\ql)\qth_i\qth_i^1+ B^{ij}\qth_i\qth_j,
\]
where $f^i$ are non-vanishing functions, $g^i = \ql^if^i$ and the functions $A^{ij}$ and $B^{ij}$ are given by
\[
A^{ij} = \frac 12\left(\frac{f^i}{f^j}\diff{f^j}{\ql^i}\ql^{j,1}-\frac{f^j}{f^i}\diff{f^i}{\ql^j}\ql^{i,1}\right),\quad B^{ij} = \frac 12\left(\frac{g^i}{f^j}\diff{f^j}{\ql^i}\ql^{j,1}-\frac{g^j}{f^i}\diff{f^i}{\ql^j}\ql^{i,1}\right).
\]
The coordinates $\ql^1,\cdots,\ql^n$ are called the canonical coordinates of $(P_0, P_1)$.

In what follows, when we consider semisimple bihamiltonian structures of hydrodynamic type, it is always assumed that we choose the canonical coordinates as local coordinates. We will use $u^1,\cdots,u^n$ instead of $\ql^1,\cdots,\ql^n$ to denote them. In these coordinates, we represent a bihamiltonian structure $(P_0, P_1)$ of hydrodynamic type as
\begin{equation}
\label{norm-p}
P_0 = \frac 12\int \sum f^i(u)\qth_i\qth_i^1+ A^{ij}\qth_i\qth_j,\quad P_1 = \frac 12\int \sum u^if^i(u)\qth_i\qth_i^1+ B^{ij}\qth_i\qth_j.
\end{equation}

Given a hydrodynamic bihamiltonian structure $(P_0,P_1)$, we construct the variational bihamiltonian cohomology as follows. Consider the space of 1-forms $\Qo$ of the infinite jet space $J^\infty(\hat M)$, locally it is an $\hm A$-module generated by $\qd u^{i,s}$ and $\qd\qth_i^s$ for $i = 1,\cdots, n$ and $s\geq 0$:
\[
\Qo = \biggl\{\sum_{i;s\geq 0}g_{i,s}\qd u^{i,s}+h^i_s\qd\qth_i^s\mid g_{i,s}, h^i_s\in\hm A\biggr\}.
\]
Each derivation of $\hm A$ induces an action on $\Qo$ by the Lie derivative (see Sect.\,\ref{vbc} for details). In particular we consider the action of Lie derivative of $\qp_x$ on $\Qo$, which we still denote by $\qp_x$, and we denote by $\bar\Qo$ the quotient space $\Qo/{\qp_x\Qo}$. We can verify that the actions on $\bar\Qo$ given by the Lie derivatives of $D_{P_0}$ and $D_{P_1}$, which we denote by $\tilde D_0$ and $\tilde D_1$ respectively, equip $\bar\Qo$ a structure of double complex. We grade the space $\bar\Qo$ similarly as we do for the space $\hm F$ by setting 
\[
\deg_x \delta u^{i,s} = \deg_x \delta\qth^{s}_i = s;\quad \deg_\qth \delta u^{i,s} = 0,\quad \deg_\qth \delta \qth_i^s = 1.
\]
Then we define the variational bihamiltonian cohomology groups as follows:
\[\vbh^p_d(\bar\Qo,\tilde D_0,\tilde D_1) = \frac{\bar\Qo_d^p\cap\ker \tilde D_0\cap\ker\tilde D_1}{\bar\Qo_d^p\cap\ima\tilde D_0\tilde D_1}.\]

We prove the following theorem in Sect.\! \ref{vbh23} and Sect.\! \ref{vsh}.
\begin{Th}
\label{vbh-res}
The variational bihamiltonian cohomology groups for a semisimple bihamiltonian structure $(P_0,P_1)$ of hydrodynamic type have the following properties:
\begin{enumerate}
\item[\rm{i)}] $\vbh^2_3(\bar\Qo,\tilde D_0,\tilde D_1)\cong\oplus_{i=1}^nC^\infty(\mathbb R);$
\item[\rm{ii)}] If $d\geq 2$ and if both $(p,d)$ and $(p+1,d)$ are NOT in the index set $I$, then $\vbh^p_d(\bar\Qo,\tilde D_0,\tilde D_1) = 0$, where the index set is defined by $I = I_1\cup I_2\cup I_3$ with
\begin{align*}
I_1 &= \{(i, j)\mid j = 0,1;\ i = j+1,\cdots,j+n+1\};\\
I_2 &= \{(i,j)\mid j = 2,\cdots,n;\ i = j,\cdots,j+n+1\};\\
I_3 &= \{(i,j)\mid j = n+1,n+2,n+3;\ i = j,\cdots,j+n\}.
\end{align*}
\item[\rm{iii)}] $\vbh^1_2(\bar\Qo,\tilde D_0,\tilde D_1) = 0.$
\end{enumerate}
\end{Th}

To illustrate the application of the variational bihamiltonian cohomology, in Sect.\! \ref{conformal} we study the conformal property of bihamiltonian structures which is shared by the bihamiltonian structures associated with Frobenius manifolds. Let us introduce the notations
\begin{align*}
&\mathrm{Der}(\hat{\mathcal A})^p:=\{X\in \mathrm{Der}(\hat{\mathcal A})\mid X(\hm A^k)\subset \hm A^{k+p}\},\\
&\mathrm{Der}(\hat{\mathcal A})^p_d:=\{X\in \mathrm{Der}(\hat{\mathcal A})\mid X(\hm A^k_l)\subset \hm A^{k+p}_{l+d}\}.
\end{align*}
Then the conformal property of a bihamiltonian structure can be described as the follows.
\begin{Def}
\label{def-conf}\label{def-2-1-zh}
A bihamiltonian structure $(P_0, P_1)$ is called conformal if there exists a nonzero derivation $E \in \mathrm{Der}(\hat{\mathcal A})^0$ and real numbers $\mu,\ql_0,\ql_1$ such that:
\begin{equation}\label{eq-def-2-1-zh}
\left[E,\qp_x\right] = \mu\qp_x,\quad \left[E, D_{P_a}\right] = \ql_aD_{P_a},\quad a = 0,1.
\end{equation}
\end{Def}

\begin{Ex}
The bihamiltonian structure of the dispersionless KdV hierarchy is given by 
\[
P_0 = \frac 12\int\qth\qth^1;\quad P_1 = \frac 12 \int u\qth\qth^1.
\]
One can check that it is a conformal bihamiltonian structure with $E$ given by
\[
E = \sum_{s\geq 0}(\ql_1-\ql_0+s\mu)u^{(s)}\diff{}{u^{(s)}}+(\ql_1+(s-1)\mu)\qth^s\diff{}{\qth^s}.
\]
\end{Ex}
\begin{Th}
\label{g0-conf}
A semisimple bihamiltonian structure $(P_0, P_1)$ of hydrodynamic type is conformal if and only if there exist real numbers $d^1,\cdots,d^n$ such that the functions $f^1,\cdots,f^n$ given in \eqref{norm-p} satisfy the following identities:
\begin{align}
\label{hom-f}
&\sum_j u^j\diff{f^i}{u^j} = d^i f^i,\quad \forall\ i;\\
\label{irre-f}
&(d^i-d^j)\diff{f^j}{u^i} = 0,\quad \forall\ i, j.
\end{align}
In such a case, if we require that the derivation $E$ has differential degree 0, then $E$ is an Euler type vector field  given by
\begin{equation}
\label{gen-e}
E = \sum_{i;s\geq 0}(\ql_1-\ql_0+s\mu)u^{i,s}\diff{}{u^{i,s}}+(\ql_1-(\ql_1-\ql_0)d^i+(s-1)\mu)\qth_i^s\diff{}{\qth_i^s}.
\end{equation}
\end{Th}

Given a semisimple and conformal bihamiltonian structure  of hydrodynamic type, we consider whether its deformations are still conformal. It turns out that there is only a certain family of deformations that preserve the conformal property, and the  equivalence classes of these deformations under Miura type transformations are parametrized by $n$ constants, as it is shown by the following theorem.
\begin{Th}
\label{g1-conf} 
Let $(P_0,P_1)$ be a semisimple and conformal bihamiltonian structure of hydrodynamic type, then its deformation $(\tilde P_0,\tilde P_1)$ is conformal if and only if the central invariants are given by:
\begin{equation}
\label{con-ci}
c_i(u^i) = C_i(u^i)^{m_i},\quad m_i = \frac{\ql_1-\ql_0-2\mu-(\ql_1-\ql_0)d^i}{\ql_1-\ql_0},
\end{equation}
where $C_i$ are arbitrary constants.
\end{Th}

\section{Variational Bihamiltonian Cohomology}\label{vbc}
In this section, we start with the definition of variational Hamiltonian cohomology for a single Hamiltonian structure of hydrodynamic type and prove its triviality. Then we define the variational bihamiltonian cohomology and make some necessary algebraic preparations for computing the cohomology groups.

\subsection{Variational forms on the infinite jet bundle.} We recall the constructions of the variational forms on the infinite jet bundle $J^\infty(\hat M)$ (cf. \cite{dubrovin2001normal} and reference therein). Let $\mathcal E$ be the space of differential forms on $J^\infty(\hat M)$, locally is just the space
\[
\mathcal E = \wedge^* \left(\mathrm{Span}_{\hm A}\left\{\qd\qth_i^s, \qd u^{i,s}\mid i = 1,\cdots,n,\, s\ge 0 \right\}\right).
\]
Here we adopt Deligne's sign rule for these variational forms, for example
\[
\qth_i^s\qd\qth_j^t = -\qd\qth_j^t\qth_i^s,\quad \qd\qth_i^s\qd\qth_j^t = \qd\qth_j^t\qd\qth_i^s,\quad \qd u^{i,s}\qd\qth_j^t = -\qd\qth_j^t\qd u^{i,s},\quad \qd u^{i,s}\qd u^{j,t} = -\qd u^{j,t}\qd u^{i,s}.
\]

Any vector field $X\in\der$ induces an action on $\mathcal E$ by the Lie derivative, which is defined by the Cartan formula
\[\mathcal L_X = \qd\comp\iota_X+\iota_X\comp\qd,\]
where the de Rham differential is given by
\[
\qd = \sum_{i;s\geq 0}\qd u^{i,s}\diff{}{u^{i,s}}+\qd\qth_i^s\diff{}{\qth_i^s}.
\]
In particular, the vector field $\qp_x$ defined in \eqref{dx-zh} induces an action on $\mathcal E$, which we still denote by $\qp_x$. We define the space of local functionals of forms by
\[\bar{\mathcal E} = \mathcal E/{\qp_x \mathcal E},\] 
and represent its elements in the form $\int f$ with $f\in\mathcal E$.
In addition, if a vector field $X\in\der$ commutes with $\qp_x$, then its Lie derivative $\mathcal L_X$ also commutes with $\qp_x$ acting on $\mathcal E$, hence it induces an action on $\bm E$ which we still denote by $\mathcal L_X$.

\begin{Ex}
\label{1-form}
The space of 1-forms is given by
\[
\mathcal E^1 = \biggl\{\sum_{i;s\geq 0}g_{i,s}\qd u^{i,s}+h^i_s\qd\qth_i^s\mid g_{i,s}, h^i_s\in\hm A\biggr\}.
\]
Let us grade $\mathcal E^1$ as we do for $\hm A$ as follows:
\begin{align*}
(\mathcal E^1)_d = \biggl\{\sum_{i;s\geq 0}g_{i,s}\qd u^{i,s}+h^i_s\qd\qth_i^s\mid g_{i,s}, h^i_s\in\hm A_{d-s}\biggr\},\\
(\mathcal E^1)^p = \biggl\{\sum_{i;s\geq 0}g_{i,s}\qd u^{i,s}+h^i_s\qd\qth_i^s\mid g_{i,s}\in\hm A^p, h^i_s\in\hm A^{p-1}\biggr\}.
\end{align*}
For any element $\qo\in\bm E^1$, we can uniquely represent it in the form
\[
\qo = \int \sum_i g_i\qd u^i+h^i\qd\qth_i,\quad g_i,h^i\in\hm A.
\]
So we can identify the space $\bm E^1$ with the space of $\hm A$-valued differential 1-forms on $\hat M$.
\end{Ex}
\begin{Ex}\label{ex-3-2-zh}
Let $X\in\der^q$ and $\qo\in(\mathcal E^1)^p$ be given by
\[\qo = \sum_{i;s\geq 0}g_{i,s}\qd u^{i,s}+h^i_s\qd\qth_i^s,\quad g_{i,s}\in\hm A^p,\ h^i_s\in\hm A^{p-1},\]
then the action of $\mathcal L_X$ on $\qo$ has the expression
\[\mathcal L_X\qo = \sum_{i;s\geq 0}X(g_{i,s})\qd u^{i,s}+(-1)^{pq}g_{i,s}\qd \left(X(u^{i,s})\right)+X(h^i_s)\qd\qth_i^s+(-1)^{(p-1)q}h^i_s\qd \left(X(\qth_i^s)\right).\]
\end{Ex}

Let us consider the space $\mathrm{Der}(\hm A)^{\qp}$ of derivations on $\hm A$ which commute with $\qp_x$. Take an element $X\in \mathrm{Der}(\hm A)^{\qp}$, since $[X,\ \qp_x] = 0$  we can regard $X$ as an $\hm A$-valued vector field on $\hat M$ instead of on $J^\infty(\hat M)$. Note that $\hat M$ admits a canonical symplectic structure
\[\varpi = \sum_i \qd u^i\wedge\qd\qth_i,\]
hence the space of the $\hm A$-valued vector fields on $\hat M$ is canonically identified with the space of $\hm A$-valued differential 1-forms on $\hat M$, which is the same as the space $\bm E^1$ as we explained in the Example \ref{1-form}. Let us write down this identification explicitly. An element $X$ of $\mathrm{Der}(\hm A)^{\qp}$ can be represented as
\[X = \sum_{i;s\geq 0}\qp_x^s\left(X(u^i)\right)\diff{}{u^{i,s}}+\qp_x^s\left(X\left(\qth_i\right)\right)\diff{}{\qth^{s}_i}.\]
Restricting it on $\hat M$, we get the following $\hm A$-valued vector field on $\hat M$ which we still denote by $X$:
\[X = \sum_i X(u^i)\diff{}{u^i}+X(\qth_i)\diff{}{\qth_i}.\]
Using the canonical symplectic structure $\varpi$, we identify $X$ with the 1-form
\[W = \iota_X \varpi = \sum_i X(u^i)\qd\qth_i-X(\qth_i)\qd u^i,\]
which corresponds to a unique element $\qo\in\bm E^1$ given by
\[\qo = \int  \sum_i X(u^i)\qd\qth_i-X(\qth_i)\qd u^i.\]
We note that for an element $X \in (\mathrm{Der}(\hm A)^{\qp})^p_d$, the corresponding element $\qo$ belongs to $(\bm E^1)^{p+1}_d$. Let us denote this correspondence by
\begin{equation}
\label{s3-t1}
\Phi: \mathrm{Der}(\hm A)^{\qp}\to \bm E^1;\quad \Phi: X\mapsto \int \iota_\varpi \left(X|_{\hat M}\right).
\end{equation}

Now let $P$ be a Hamiltonian structure of hydrodynamic type. By using the identity \eqref{comm-d}, we see that the space $\mathrm{Der}(\hm A)^{\qp}$ becomes a cochain complex with the differential given by adjoint action of $D_P$. The question is that when we identify elements in $\mathrm{Der}(\hm A)^{\qp}$ as  elements in $\bm E^1$, then how the action of $D_P$ induces a differential on $\bm E^1$?

\begin{Lem}
Let $P$ be a Hamiltonian structure of hydrodynamic type, then the following identity holds true for any $X \in \mathrm{Der}(\hm A)^{\qp}$:
\begin{equation}
\label{s3-t2}
\Phi\left([D_P,\ X]\right) = \mathcal L_{D_P}\Phi(X).
\end{equation}
\end{Lem}
\begin{proof}
Since both the definition \eqref{s3-t1} and the identity \eqref{s3-t2} are coordinate free, we can choose a system of local coordinates $(v^1,\cdots, v^n; \phi_1,\cdots,\phi_n)$ on $\hat M$ such that the Hamiltonian structure $P$ has the expression \cite{dubrovin1996hamiltonian}
\[P = \frac 12\int\eta^{\qa\qb}\phi_\qa\phi_\qb^1,\]
here $(\eta^{\qa\qb})$ is a constant non-degenerate matrix, and summation over repeated lower and upper Greek indices is assumed. We call such coordinates the flat coordinates of $P$.

Now we assume that $X$ is an element of $\mathrm{Der}(\hm A)^{\qp}$ with super degree $p$, which means that $X(\hm A^q)\subset \hm A^{p+q}$, then it is straightforward to show that
\begin{equation*}[D_P, X]v^\qa = D_P\left(X(v^\qa)\right)+(-1)^{p+1}\eta^{\qa\qb}(X(\phi_\qb^1)),\quad [D_P, X]\phi_\qa = D_P(X(\phi_\qa)).\end{equation*}
Then we arrive at
\[
\Phi\left([D_P, X]\right) = \int\left(D_P\left(X(v^\qa)\right)+(-1)^{p+1}\eta^{\qa\qb}(X(\phi_\qb^1))\right)\qd\phi_\qa-D_P(X(\phi_\qa))\qd v^\qa.
\]
On the other hand we have
\begin{align*}
\mathcal L_{D_P}\Phi(X) &= \mathcal L_{D_P}\int X(v^\qa)\qd\phi_\qa-X(\phi_\qa)\qd v^\qa\\
&=\int D_P\left(X(v^\qa)\right)\qd\phi_\qa-D_P(X(\phi_\qa))\qd v^\qa-(-1)^{p+1}X(\phi_\qa)\qd\left(\eta^{\qa\qb}\phi_\qb^1\right)\\
&= \int\left(D_P\left(X(v^\qa)\right)+(-1)^{p+1}\eta^{\qa\qb}(X(\phi_\qb^1))\right)\qd\phi_\qa-D_P(X(\phi_\qa))\qd v^\qa.
\end{align*}
Therefore we prove the lemma.
\end{proof}
\subsection{Variational Hamiltonian cohomology and its triviality}
From now on we will use $\Qo$ to denote the space of 1-froms $\mathcal E^1$ and $\bar\Qo$ to denote its quotient space $\bar{\mathcal E}^1$. The space of homogeneous elements with differential degree $d$ and super degree $p$ will be denoted by $\Qo^p_d$ and $\bar\Qo^p_d$ respectively.

Let $P$ be a Hamiltonian structure of hydrodynamic type, we will use $\tilde D_P$ to denote the action $\mathcal L_{D_P}$ on the space $\Qo$ and $\bar\Qo$. By using the identity \eqref{s3-t2} we conclude that $\tilde D_P\comp\tilde D_P = 0$, so $\tilde D_P$ is a differential on the spaces $\Qo$ and $\bar\Qo$.
\begin{Def}
The variational Hamiltonian cohomology of $\Qo$ (and of $\bar \Qo$, respectively) is defined to be the cohomology of the complex $(\Qo^\bullet,\tilde D_P)$ (and of $(\bar\Qo^\bullet,\tilde D_P), respectively$) given by
\[
H^p_d(\Qo,\tilde D_P) = \frac{\Qo^p_d\cap\ker\tilde D_P}{\Qo^p_d\cap\ima\tilde D_P};\quad H^p_d(\bar\Qo,\tilde D_P) = \frac{\bar\Qo^p_d\cap\ker\tilde D_P}{\bar\Qo^p_d\cap\ima\tilde D_P}.
\]
\end{Def} 

By using the fundamental facts of the homological algebra we have the following lemma.
\begin{Lem}
\label{s3-t3}
The short exact sequence
\xym{
   0\ar[r]&\Qo\ar[r]^{\qp_x}&\Qo\ar[r]^{\pi}&\bar\Qo\ar[r] & 0}
induces the following long exact sequence of the cohomology groups for $d\geq 1$:
\[
\cdots \to H^p_d(\Qo,\tilde D_P)\to H^p_d(\bar\Qo,\tilde D_P)\to H^{p+1}_d(\Qo,\tilde D_P)\to\cdots.
\]
\end{Lem}
\begin{Rem}
On the space $\hm A$, the map $\qp_x$ has the kernel $\mathbb R$, however on the space $\Qo$ the map $\qp_x$ is injective.
\end{Rem}

\begin{Th}[Triviality of the variational Hamiltonian cohomology]
We have 
\[H^p_d(\Qo,\tilde D_P) = 0\]
for $p>0,\ d>0$.
\end{Th}
\begin{proof}
We choose locally a system of flat coordinates $(v^1,\cdots, v^n; \phi_1,\cdots,\phi_n)$ on $\hat M$ such that $P = \frac 12\int \eta^{\qa\qb}\phi_\qa\phi_\qb^1$. Then for a 1-form
\begin{equation*}
\qo = \sum_{s\geq 0} f_{\qa,s}\qd v^{\qa,s}+g^\qa_s\qd\phi_\qa^s\in\Qo^p_d
\end{equation*}
we have
\begin{equation*}
\tilde D_P\qo = \sum_{s\geq 0} D_P(f_{\qa,s})\qd v^{\qa,s}+(-1)^p\eta^{\qa\qb}f_{\qa,s}\qd\phi_\qb^{s+1}+D_P(g^\qa_s)\qd\phi_\qa^s.
\end{equation*}
If $\qo\in\ker \tilde D_P$, then we see that
\begin{equation*}
D_P(g^\qa_0)= 0;\quad D_P(g_{s+1}^\qa)+(-1)^p\eta^{\qa\qb}f_{\qb,s} = 0,\ s\geq 0,
\end{equation*}
so we can write $\qo$ in the form
\begin{equation*}
\qo = \tilde D_P \left(\sum_{s\geq 0}(-1)^{p-1}\eta_{\qa\qb}g^\qa_{s+1}\qd v^{\qb,s}\right)+ g^\qa_0\qd\phi_\qa.
\end{equation*}
For $g^\qa_0\in\hm A^{p-1}_d$, by using the triviality of the Hamiltonian cohomology (see, for example \cite{liu2018lecture,liu2011jacobi}) we can represent it as $g^\qa_0 = D_P(h^\qa_0)$. Thus the cocycle $\qo$ must also be a coboundary. The theorem is proved.
\end{proof}
From Lemma \ref{s3-t3} we have the following corollary.
\begin{Cor}
We have $H^p_d(\bar\Qo,\tilde D_P) = 0$ for $p>0,\ d>0$.
\end{Cor}

\subsection{Definition of the variational bihamiltonian cohomology}
We proceed to define the variational bihamiltonian cohomology and discuss its relations with the bihamiltonian cohomology. Let $(P_0,P_1)$ be a semisimple bihamiltonian structure of hydrodynamic type and $u^1,\cdots,u^n$ be its canonical coordinates. We will use $\tilde D_0$ and $\tilde D_1$ to denote $\tilde D_{P_0}$ and $\tilde D_{P_1}$ respectively. By using the identity \eqref{s3-t2} we have $\tilde D_i\tilde D_j+\tilde D_j\tilde D_i = 0$ for $i,j = 0,1$.
\begin{Def}
The variational bihamiltonian cohomology for $(P_0,P_1)$ is defined to be the following groups:
\begin{align}
\vbh^p_d(\Qo,\tilde D_0,\tilde D_1) = \frac{\Qo_d^p\cap\ker \tilde D_0\cap\ker\tilde D_1}{\Qo_d^p\cap\ima\tilde D_0\tilde D_1};\\ 
\vbh^p_d(\bar\Qo,\tilde D_0,\tilde D_1) = \frac{\bar\Qo_d^p\cap\ker \tilde D_0\cap\ker\tilde D_1}{\bar\Qo_d^p\cap\ima\tilde D_0\tilde D_1}.
\end{align}
\end{Def}

The following lemmas are important for computing the cohomology groups.
\begin{Lem}
\label{lem-lambda}
The cohomology groups of the cochain complex $\Qo[\ql] = \Qo\otimes \mathbb R[\ql]$ with differential $\qp_\ql = \tilde D_1-\ql\tilde D_0$ is isomorphic to the variational bihamiltonian cohomology, i.e.
\begin{equation}
{H}^p_d(\Qo[\ql],\qp_\ql)\cong \vbh^p_d(\Qo,\tilde D_0,\tilde D_1),\quad d\geq 2.
\end{equation}
Similarly, we have the following isomorphisms for the corresponding quotient spaces:
\begin{equation}
{H}^p_d(\bar\Qo[\ql],\qp_\ql)\cong \vbh^p_d(\bar\Qo,\tilde D_0,\tilde D_1),\quad d\geq 2.
\end{equation}
\end{Lem}
\begin{proof}
The lemma follows from the triviality of the variational hamiltonian cohomology. For details, one may refer to the proof of Lemma 4.4 of \cite{liu2013bihamiltonian}.
\end{proof}

\begin{Lem}[Salamander lemma]
\label{sal}
 We have the following isomorphism induced by $\tilde D_0$ for $p>0, d>0$:
\begin{align*}
\tilde D_0:\quad \frac{\bar\Qo^p_d\cap\ker \tilde D_1\tilde D_0}{\bar\Qo^p_d\cap(\ima\tilde D_0+\ima\tilde D_1)}\xrightarrow[]{\cong} \vbh^{p+1}_{d+1}(\bar\Qo,\tilde D_0,\tilde D_1).
\end{align*}
\end{Lem}
\begin{proof}
The lemma follows easily from the triviality of the variational Hamiltonian cohomology.
\end{proof}

The definition of the variational bihamiltonian cohomology is comparable with the definition of bihamiltonian cohomology. On the space $\hm A$ there are differentials $D_0$ and $D_1$ which are defined in \eqref{def-d} by the bihamiltonian structure $(P_0, P_1)$, and on the space $\hm F$ there are the induced differential which are denoted by $d_0$ and $d_1$. The bihamiltonian cohomology is then given by
\begin{align*}
\bh^p_d(\hm A,D_0,D_1) = \frac{\hm A_d^p\cap\ker  D_0\cap\ker D_1}{\hm A_d^p\cap\ima D_0 D_1}\cong H^p_d(\hm A[\ql],D_1-\ql D_0);\\ \bh^p_d(\hm F,d_0,d_1) = \frac{\hm F_d^p\cap\ker d_0\cap\ker d_1}{\hm F_d^p\cap\ima d_0d_1}\cong H^p_d(\hm F[\ql],d_1-\ql d_0).
\end{align*}

There is a natural cochain map 
\begin{equation}
\label{bhvbh}
\qd:\quad \hm F[\ql]\to \bar\Qo[\ql]
\end{equation}
between the complexes $\hm A[\ql]$ and $\Qo[\ql]$ given by the de Rham differential $\qd$.
This map can be further illustrated as follows: We first note that the space $\hm F$ can be identified with 
\[
\mathrm{Der}(\hm A)^D:=\{X\in\mathrm{Der}(\hm A)\mid \exists\ Y\in\hm F, X = D_Y\}
\]
by using the map \eqref{def-d}. By applying the identity \ref{comm-d}, we see that the differentials $d_0, d_1$ on $\hm F$ induce differentials $\mathrm{ad}_{D_0}, \mathrm{ad}_{D_1}$ on $\mathrm{Der}(\hm A)^D$. Since it is obvious that $\mathrm{Der}(\hm A)^D\subseteq \mathrm{Der}(\hm A)^\qp$, the identity \eqref{s3-t2} shows that the map \eqref{bhvbh} can be viewed as a natural embedding (with a change of signs, see the remark below):
\[i:\mathrm{Der}(\hm A[\ql])^D\hookrightarrow \mathrm{Der}(\hm A[\ql])^\qp.\]
Therefore we conclude that the bihamiltonian cohomology can be viewed as the cohomology on the space $\mathrm{Der}(\hm A[\ql])^D$, while the variational bihamiltonian cohomology is the cohomology on the space $\mathrm{Der}(\hm A[\ql])^\qp$. In this sense, the bihamiltonian cohomology is just a restriction of the variational bihamiltonian cohomology onto a subcomplex.

\begin{Rem}
The map $\qd: \hm F\to\bar \Qo$ can be viewed as a generalization of the correspondence between the hamiltonian functions and the hamiltonian vector field on a finite dimensional symplectic manifold. In our case, this correspondence is twisted by a sign:
\[
\qd X = (-1)^{p-1}\Phi(D_X),\quad X\in\hm F^p.
\] 
\end{Rem}

\section{The cohomology group $\vbh^2_3$}
\label{vbh23}

In this section, we compute the bihamiltonian cohomology group $\vbh^2_3(\bar\Qo,\tilde D_0,\tilde D_1)$. The computation of other cohomology groups will be covered in the next section. We fix a semisimple bihamiltonian $(P_0,P_1)$ of hydrodynamic type and work in the canonical coordinates such that the bihamiltonian structure is given by \eqref{norm-p}.
Note that if we define a derivation
\begin{align*}
D(g^1,\cdots,g^n) = &\sum_{s\geq 0;i}\qp_x^s(g^i\qth_i^1)\diff{}{u^{i,s}}
\\&+\frac{1}{2}\sum_{s\geq 0;i,j}\qp_x^s\left(\qp_jg^iu^{j,1}\qth_i+g^i\frac{\qp_ig^j}{g^j}u^{j,1}\qth_j-g^j\frac{\qp_jg^i}{g^i}u^{i,1}\qth_j\right)\diff{}{u^{i,s}}\\
&+\frac{1}{2}\sum_{s\geq 0;i,j}\qp_x^s\left(\qp_ig^j\qth_j\qth_j^1+g^j\frac{\qp_jg^i}{g^i}\qth_i\qth_j^1-g^j\frac{\qp_jg^i}{g^i}\qth_j\qth_i^1\right)\diff{}{\qth_i^s}\\
&+\frac{1}{2}\sum_{s\geq 0;i,j,k}\qp_x^s\left(\qp_i\left(g^k\frac{\qp_kg^j}{g^j}\right)u^{j,1}\qth_k\qth_j-\qp_j\left(g^k\frac{\qp_kg^i}{g^i}\right)u^{j,1}\qth_k\qth_i\right)\diff{}{\qth_i^s},
\end{align*}
for a set of functions $g^1,\cdots,g^n$, then we know that
\[
D_{P_0} = D(f^1,\dots,f^n),\quad D_{P_1} = D(u^1f^1,\dots,u^nf^n).
\]
Here and henceforth we will use $\qp_i$ to denote $\diff{}{u^i}$. 

From Lemma \ref{sal} it follows that in order to compute the cohomology group $\vbh^2_3(\bar\Qo,\tilde D_0,\tilde D_1)$ we only need to consider the spaces 
\[\mathcal Z:= \bar\Qo^1_2\cap \ker(\tilde D_0\tilde D_1),\quad
\mathcal B:= \bar\Qo^1_2\cap(\ima\tilde D_0+\ima\tilde D_1).\]
For an element $\qo\in\bar \Qo^1_2$, we can represent it uniquely in the form
\begin{equation}
\label{s3-t4}
\qo = \int \sum_i X^i\qd u^i+Y^i\qd\qth_i,\quad X^i\in\hm A^1_2,\quad Y^i\in\hm A^0_2,
\end{equation}
where the differential polynomials can be written as
\begin{align}
\label{s3-t5}
X^i &= \sum_j X^{(i)}_{j}\qth_j^2+\sum_{j,k}\left(X^{(i)}_{kj}u^{j,1}\qth_k^1+Z^{(i)}_{jk}u^{k,2}\qth_j\right)+\sum_{j,k,l}Z^{(i)}_{j;kl}u^{k,1}u^{l,1}\qth_j;\\
\label{s3-t6}
Y^i &= \sum_j Y^{(i)}_ju^{j,2}+\sum_{j,k}Y^{(i)}_{jk}u^{j,1}u^{k,1}.
\end{align}
\begin{Def}\label{ind-zh}
For an element $\qo\in\bar\Qo^1_2$ which is represented in the form \eqref{s3-t4}--\eqref{s3-t6}, the indices $ind_i(\qo)$ for $i = 1,\cdots, n$ with respect to a semisimple bihamiltonian structure $(P_0,P_1)$ of hydrodynamic type are defined to be the functions
\[
ind_i(\qo):=\frac{1}{f^i}(X^{(i)}_i+Y^{(i)}_i).
\]
\end{Def}
We will see later that the indices defined above are generalizations of the central invariants of deformations of bihamiltonian structures of hydrodynamic type. To compute the cohomology group $\vbh^2_3(\bar\Qo)\cong \mathcal Z/\mathcal B$, we need the following two lemmas.
\begin{Lem}
\label{s3-t7}
For any given cocycle $\qo\in\mathcal Z$, the index $ind_i(\qo)$ is a function of single variable $u^i$ for any $i=1,\cdots,n$.
\end{Lem}
\begin{Lem}
\label{s3-t8}
A cocycle $\qo\in\mathcal Z$ is a coboundary, i.e. $\qo\in\mathcal B$, if and only if  
\[ind_i(\qo) = 0,\quad i = 1,\cdots,n.\]
\end{Lem}
\begin{Th}
\label{h23}
The quotient space $\mathcal Z/\mathcal B\cong\oplus_{i=1}^nC^\infty(\mathbb R)$.
\end{Th}
\begin{proof}
Given $n$ functions of single variable $c_1(u^1),c_2(u^2),\cdots,c_n(u^n)$, we construct an element $\qt\in\mathcal Z$ as follows:
\[
\qt = \int \qd\biggl(D_1\sum_i c_i(u^i)u^{i,1}\log u^{i,1}-D_0\sum_i u^ic_i(u^i)u^{i,1}\log u^{i,1}\biggr).
\]
One can check directly, or refer to \cite{liu2005deformations}, to confirm the fact that actually $\qt$ is a 1-form with differential polynomial coefficients and that $\qt\in\mathcal Z$.

By a straightforward computation we can show that 
\[ind_i(\qt) = -3c_i(u^i),\quad i = 1,\cdots,n.\]
Therefore, it follows from Lemma \ref{s3-t7} that we can choose suitable $c_i(u^i)$ such that 
\[ind_i(\qo-\qt) = 0\] for any given cocycle $\qo\in\mathcal Z$. Then we apply Lemma \ref{s3-t8} to conclude that the class $[\qo]$ coincides with the class $[\qt]$ in the space $\mathcal Z/\mathcal B$. Lemma \ref{s3-t8} also implies that $\qt$ is a coboundary if and only if $c_i(u^i) = 0$, hence different choices of functions $c_1(u^1),\cdots,c_n(u^n)$ give different classes $[\qt]$ in $\mathcal Z/\mathcal B$. The theorem is proved.
\end{proof}

The remaining part of this section is devoted to the proof of Lemma \ref{s3-t7} and Lemma \ref{s3-t8}. The strategy of our proof is: We first prove the `only if' part of Lemma \ref{s3-t8}, then we prove Lemma \ref{s3-t7}. Finally we prove the `if' part of Lemma \ref{s3-t8}.

\begin{proof}[\textbf{Proof of Lemma \ref{s3-t8} (Part 1)}] We are to show that if $\qo\in\mathcal B$ then its indices vanish. Take $\qa\in\bar\Qo^0_1$ and write it uniquely as
\[
\qa = \int \sum_{i,j}\qa^{(i)}_ju^{j,1}\qd u^i,\]
then we have
\begin{align}
\label{s3-t9}
\tilde D_0\qa &= \int \sum_{i,j}D_0(\qa^{(i)}_ju^{j,1})\qd u^i+\qa^{(i)}_ju^{j,1}\qd(D_0(u^i))\\
\nonumber
&=\int \sum_{i,j}\left(\qa^{(i)}_jf^j\qth_j^2+\cdots\right)\qd u^i+\left(-\qa^{(i)}_jf^iu^{j,2}+\cdots\right)\qd\qth_i.
\end{align}
Here $\cdots$ stands for the terms that make no contribution to the index. Then from Definition \ref{ind-zh} it follows that $ind_i(\tilde D_0\qa) = 0$. Similarly we have $ind_i(\tilde D_1\qa) = 0$. Hence the indices of a coboundary must vanish.
\end{proof}

The above proof shows that the indices can be defined for a class $[\qo]\in\mathcal Z/\mathcal B$. To prove Lemma \ref{s3-t7}, we first choose a `normal form' for every class $[\qo]\in\mathcal Z/\mathcal B$ that will simplify the computation.

\begin{Lem}
\label{s3-t10}
For any given class $[\qs]\in\mathcal Z/\mathcal B$, there exists a unique $\qo\in\mathcal Z$ which can be represented in the form \eqref{s3-t4}--\eqref{s3-t6} and satisfies the following conditions:
\begin{enumerate}
\item[\rm(1)] $[\qo] = [\qs]$;
\item[\rm(2)] $X^{(i)}_j = 0$ for $j\neq i$;
\item[\rm(3)] $Y^{(i)}_j = 0$ for any $i,j$;
\item[\rm(4)] $Y^{(i)}_{ii} = 0$.
\end{enumerate}
Such a form $\qo$ is called the normal form of the class $[\qs]$.
\end{Lem}
\begin{proof}
We first take $\qo = \qs$, then the first condition is satisfied. We then adjust $\qo$ by adding elements of $\mathcal B$ such that other conditions are also satisfied. Firstly, according to \eqref{s3-t9}, we change $\qo$ to $\tilde\qo = \qo+\tilde D_0\qg$ with
\[
\qg = \int \sum_{i,j}\frac{Y^{(i)}_j}{f^i}u^{j,1}\qd u^i,
\]
then the third condition is satisfied. As for other conditions, we want to find $\qa,\qb\in\bar \Qo^0_1$ such that $\tilde\qo+\tilde D_0\qa+\tilde D_1\qb$ satisfies all the four conditions.

Let us write $\qa = \int\sum_i\qa_i\qd u^i$ and $\qb = \int\sum_i\qb_i\qd u^i$ for some $\qa_i,\qb_i\in\hm A^0_1$. We can uniquely represent $\tilde D_0\qa$ and $\tilde D_1\qb$ as follows:
\begin{equation}
\tilde D_0\qa = \int \sum_i A_i\qd u^i+W_i\qd\qth_i;\quad\tilde D_1\qb = \int \sum_i B_i\qd u^i+R_i\qd\qth_i,
\end{equation}
where the differential polynomials $W_i$ and $R_i$ are given by
\begin{align*}
W_i =& -(\qa_if^i)^\prime+\frac 12\sum_j\left(\qa_i\aij ji u^{j,1}+\qa_j\bij ji u^{i,1}-\qa_j\bij ij u^{j,1}\right);\\ 
R_i =& -(\qb_iu^if^i)^\prime+\frac 12\sum_j\left(\qb_iu^i\aij ji u^{j,1}+\qb_ju^j\bij ji u^{i,1}-\qb_ju^i\bij ij u^{j,1}\right)\\
&+\frac 12 \qb_if^iu^{i,1}.
\end{align*}
Note that we should make sure that the third condition is satisfied, so if we further write $\qa_i = \sum_j\qa^{(i)}_ju^{j,1}$ and $\qb_i = \sum_j\qb^{(i)}_ju^{j,1}$, and compare the coefficients of $u^{j,2}$ of $W_i$ and $R_i$, we arrive at $u^i\qb^{(i)}_{j}+\qa^{(i)}_{j} = 0$, hence we must take $\qa_i = -u^i\qb_i$. Since the coefficients of $(u^{i,1})^2$ in $W_i+R_i$ are given by $\qb_{i}^{(i)}f^i$, we can choose suitable functions $\qb_{i}^{(i)}$ such that the condition (4) is satisfied.
Finally we compute the coefficients of $\qth_j^2$ in $A_i+B_i$ to obtain
\[
\qb^{(i)}_j(u^j-u^i)f^j,
\]
so for $j\neq i$ we can choose suitable functions $\qb^{(i)}_j$ such that the condition (2) is satisfied. Thus we find a form $\qo$ which satisfy all the four conditions, and the above computation also shows that such an $\qo$ is unique. The lemma is proved.
\end{proof}

As a byproduct of the above computation, we have the following theorem. 
\begin{Th}
\label{vbh12}
We have $\vbh^1_2(\bar\Qo,\tilde D_0,\tilde D_1) = 0$.
\end{Th} 
\begin{proof}
Recall that by definition $\vbh^1_2(\bar\Qo,\tilde D_0,\tilde D_1) = \bar\Qo^1_2\cap\ker\tilde D_0\cap\ker\tilde D_1$. Now take any cocycle $\qo$, by using the triviality of the variational Hamiltonian cohomology we can find $\qa, \qb\in \bar\Qo^0_1$ such that 
\begin{equation}
\qo = \tilde D_0\qa = \tilde D_1\qb.
\end{equation}
Let us show that $\qa=\qb=0$. To this end we write $\qa = \int\sum_i\qa_i\qd u^i$ and $\qb = \int\sum_i\qb_i\qd u^i$ for some $\qa_i,\qb_i\in\hm A^0_1$, and represent $\tilde D_0\qa$, $\tilde D_1\qb$, as we do in the proof of Lemma \ref{s3-t10}, in the form
\begin{equation}
\tilde D_0\qa = \int \sum_i A_i\qd u^i+W_i\qd\qth_i;\quad\tilde D_1\qb = \int \sum_i B_i\qd u^i+R_i\qd\qth_i.
\end{equation}
From the coefficients of $u^{j,2}$ in $W_i$ and $R_i$ it follows that $\qa_i = u^i\qb_i$, and from the coefficients of $(u^{i,1})^2$ in $W_i$ and $R_i$ we conclude that $\qb^{(i)}_i = 0$.
Finally, from the coefficients of $\qth^{2}_j$ in $A_i$ and $B_i$ for $j\neq i$ it follows that $\qb^{(i)}_j = 0$ when $j\neq i$. Therefore $\qa = \qb = 0$ and the theorem is proved.
\end{proof}

Now let us come back to prepare the proof of Lemma \ref{s3-t7}. Take a class $[\qo]\in\mathcal Z/\mathcal B$ with $\qo$ being its normal form which can be represented as
\begin{equation}
\label{norm-1}
\qo = \int \sum_i X^i\qd u^i+Y^i\qd\qth_i,\quad X^i\in\hm A^1_2,\quad Y^i\in\hm A^0_2,
\end{equation}
where the differential polynomials can be written in the form
\begin{align}
\label{norm-2}
X^i &=  X^{(i)}_{i}\qth_i^2+\sum_{j,k}\left(X^{(i)}_{kj}u^{j,1}\qth_k^1+Z^{(i)}_{jk}u^{k,2}\qth_j\right)+\sum_{j,k,l}Z^{(i)}_{j;kl}u^{k,1}u^{l,1}\qth_j;\\
\label{norm-3}
Y^i &= \sum_{j,k}Y^{(i)}_{jk}u^{j,1}u^{k,1},\quad Y^{(i)}_{ii} = 0.
\end{align}
Further more, we require that
\begin{equation}
\label{norm-4}
Z^{(i)}_{j;kl} = Z^{(i)}_{j;lk},\quad Y^{(i)}_{jk}=Y^{(i)}_{kj}.\end{equation}
Due to the part of Lemma \ref{s3-t8} that we just proved, different representatives of a class in $\mathcal Z/\mathcal B$ have the same indices, hence to prove Lemma \ref{s3-t7} we only need to show that  $ind_i(\qo) = X^{(i)}_i/f_i$ is a function of $u^i$ for each $i=1,\dots,n$.

Since $\qo\in\mathcal Z$, from the triviality of variational Hamiltonian cohomology it follows the existence of $\qa\in\bar\Qo^1_2$ such that $\tilde D_0\qo = \tilde D_1\qa$. Such an $\qa$ is unique up to the addition of an image of $\tilde D_1$, hence we can make a particular choice of $\qa$ in a similar way as we choose the normal form of $\qo$. More explicitly, it is not difficult to see that there exists a unique $\qa$ such that
\begin{equation}
\label{norm-5}
\qa = \int \sum_i P^i\qd u^i+Q^i\qd\qth_i,\quad P^i\in\hm A^1_2,\quad Q^i\in\hm A^0_2,
\end{equation}
where the differential polynomials can be written as
\begin{align}
\label{norm-6}
P^i &=  \sum_j P^{(i)}_{j}\qth_j^2+\sum_{j,k}\left(P^{(i)}_{kj}u^{j,1}\qth_k^1+W^{(i)}_{jk}u^{k,2}\qth_j\right)+\sum_{j,k,l}W^{(i)}_{j;kl}u^{k,1}u^{l,1}\qth_j,\\
\label{norm-7}
Q^i &= \sum_{j,k}Q^{(i)}_{jk}u^{j,1}u^{k,1}
\end{align}
with
\begin{equation}
\label{norm-8}
W^{(i)}_{j;kl} = W^{(i)}_{j;lk},\quad Q^{(i)}_{jk}=Q^{(i)}_{kj}.\end{equation}

\begin{proof}[\textbf{Proof of Lemma \ref{s3-t7}}]
Consider a class $[\qo]\in\mathcal Z/\mathcal B$ where $\qo$ is given by the normal form \eqref{norm-1}--\eqref{norm-4}. Let $\qa\in\bar\Qo^1_2$ be the unique element given by \eqref{norm-5}--\eqref{norm-8} such that $\tilde D_0\qo = \tilde D_1\qa$. For the bihamiltonian structure $(P_0, P_1)$ given in \eqref{norm-p}, we denote
\begin{equation}
\label{aij}
a_{ij} = \frac{1}{2}\qp_if^j,\quad b_{ij} = \frac{1}{2}f^i\frac{\qp_if^j}{f^j}.
\end{equation}
Let us first compute the differential polynomials $M^i,N^i,S^i,T^i$ that are defined by
\[
\tilde D_0\qo = \int \sum_i M^i\qd u^i+N^i\qd\qth_i,\quad \tilde D_1\qa = \int \sum_i S^i\qd u^i+T^i\qd\qth_i.
\]
In what follows of the proof, we will omit the symbol of summations and use $j,k,l$ to denote the indices that should be summed over $1,\cdots, n$. The index $i$ is a fixed index and do not participate in the summation. It is straightforward to obtain
\begin{align*}
M^i &= D_0(X^i)-\qp_if^jX^j\qth_j^1-X^k\qp_ia_{jk}u^{j,1}\qth_k+\left(X^ja_{ij}\qth_j\right)'-X^k\qp_ib_{kj}u^{j,1}\qth_j\\
&+\left(X^jb_{ji}\qth_i\right)'+X^k\qp_ib_{jk}u^{k,1}\qth_j-\lrb{X^ib_{ji}\qth_j}'+Y^k\qp_ia_{kj}\qth_j\qth_j^1
+Y^k\qp_ib_{jk}\qth_k\qth_j^1\\
&-Y^k\qp_ib_{jk}\qth_j\qth_k^1+Y^l\qp_i\qp_lb_{kj}u^{j,1}\qth_k\qth_j-\left(Y^j\qp_jb_{ki}\qth_k\qth_i\right)'\\
&-Y^l\qp_i\qp_jb_{kl}u^{j,1}\qth_k\qth_l+\left(Y^j\qp_ib_{kj}\qth_k\qth_j\right)',
\end{align*}
\begin{align*}
S^i &= D_1(P^i)-P^j\qp_i\lrb{u^jf^j}\qth_j^1-P^k\qp_i\kk{u^ka_{jk}}u^{j,1}\qth_k+\kk{P^ju^ja_{ij}\qth_j}'\\
&-P^k\qp_i\kk{u^kb_{kj}}u^{j,1}\qth_j+\kk{P^ju^jb_{ji}\qth_i}'+P^k\qp_i\kk{u^jb_{jk}}u^{k,1}\qth_j-\kk{P^iu^jb_{ji}\qth_j}'\\
&-P^ja_{ij}u^{j,1}\qth_j+\kk{\frac 12 P^if^i\qth_i}'+Q^k\qp_i\kk{u^ja_{kj}}\qth_j\qth_j^1+Q^k\qp_i\kk{u^jb_{jk}}\qth_k\qth_j^1\\
&-Q^k\qp_i\kk{u^jb_{jk}}\qth_j\qth_k^1+Q^ja_{ij}\qth_j\qth_j^1+Q^l\qp_i\qp_l\kk{u^kb_{kj}}u^{j,1}\qth_k\qth_j\\
&-\kk{Q^j\qp_j\kk{u^kb_{ki}}\qth_k\qth_i}'-Q^l\qp_i\qp_j\kk{u^kb_{kl}}u^{j,1}\qth_k\qth_l+\kk{Q^j\qp_i\kk{u^kb_{kj}}\qth_k\qth_j}',\\
N^i &= D_0(Y^i)+\kk{X^if^i}'-X^ia_{ji}u^{j,1}-X^jb_{ji}u^{i,1}+X^jb_{ij}u^{j,1}-\kk{Y^ja_{ji}\qth_i}'\\
&-Y^ja_{ji}\qth_i^1-\kk{Y_jb_{ij}\qth_j}'-Y^ib_{ji}\qth_j^1+\kk{Y^ib_{ji}\qth_j}'+Y^jb_{ij}\qth_j^1\\
&+Y^j\qp_jb_{ki}u^{i,1}\qth_k-Y^k\qp_kb_{ij}u^{j,1}\qth_j-Y^i\qp_jb_{ki}u^{j,1}\qth_k+Y^k\qp_jb_{ik}u^{j,1}\qth_k,\\
T^i &= D_1(Q^i)+\kk{P^iu^if^i}'-P^iu^ia_{ji}u^{j,1}-P^ju^jb_{ji}u^{i,1}+P^ju^ib_{ij}u^{j,1}-\frac 12 P^if^iu^{i,1}\\
&-\kk{Q_ju^ia_{ji}\qth_i}'-Q^ju^ia_{ji}\qth_i^1-\kk{Q^ju^ib_{ij}\qth_j}'-Q^iu^jb_{ji}\qth_j^1+\kk{Q_iu^jb_{ji}\qth_j}'\\
&+Q^ju^ib_{ij}\qth_j^1+Q^j\qp_j\kk{u^kb_{ki}}u^{i,1}\qth_k-Q^k\qp_k\kk{u^ib_{ij}}u^{j,1}\qth_j\\
&-Q^i\qp_j\kk{u^kb_{ki}}u^{j,1}\qth_k+Q^k\qp_j\kk{u^ib_{ik}}u^{j,1}\qth_k-\kk{\frac 12 Q^if^i\qth_i}'-\frac 12 Q^if^i\qth_i^1.
\end{align*}

By comparing the coefficients of $\qth_j^3$ of both sides of the equation $N^i = T^i$ we obtain 
\begin{equation}
\label{cond-p1-1}
P^{(i)}_j = 0,\quad i\neq j;\quad X^{(i)}_i = u^iP^{(i)}_i.
\end{equation}
Comparing the coefficients of $u^{k,3}\qth_j$ on both sides of the equation $N^i = T^i$, we conclude that
\ee{
\label{cond-p1-2}
Z^{(i)}_{jk} = u^iW^{(i)}_{jk},\quad\forall\ i,j,k.
}

Next, from the coefficients of $\qth_j^1\qth_i^2$ in $M^i$ and $S^i$ for $j\neq i$, and from \eqref{cond-p1-1} it follows that
\ee{
\label{cond-p5}
\kk{\qp_jP^{(i)}_if^j-2P^{(i)}_ib_{ji}}(u^i-u^j) = \kk{X^{(i)}_{ji}-P^{(i)}_{ji}u^i}f^i,\quad i\neq j.
}
Similarly, from the coefficients of $u^{j,2}\qth_k^1$ in $N^i$ and $T^i$, and from the identity \eqref{cond-p1-2} it follows that
\ee{
\label{cond-p8-3}
X^{(i)}_{kj} = u^iP^{(i)}_{kj},\quad\forall\ i,j,k.
}
The identity \eqref{cond-p5} and \eqref{cond-p8-3} lead to
\[\qp_j\frac{P^{(i)}_i}{f^i} = 0,\quad i\neq j.\]
Thus for each $i$ the function
\[
ind_i(\qo) = \frac{X^{(i)}_i}{f^i} = u^i\frac{P^{(i)}_i}{f^i}.
\]
depends only on $u^i$. The lemma is proved.
\end{proof}

Finally, let us complete the proof of Lemma \ref{s3-t8}. Before proceeding to the proof, we first derive some relations satisfied by the coefficients in $\qo$ and $\qa$. In the rest of this section,  we will continue to use the notations which are used in the proof of Lemma \ref{s3-t7}.

By comparing the coefficients of $\qth_k\qth_j^3$, $\qth_j\qth_j^3$ and $\qth_i\qth_j^3$ in $M^i$ and $S^i$ and by using the relations \eqref{cond-p1-1} and \eqref{cond-p1-2} we obtain
\ee{
\label{cond-p2-1}
Z^{(i)}_{kj} = W^{(i)}_{kj} = 0,\quad i\neq j\neq k\neq i.
}
\ee{
\label{cond-p2-2}
P^{(i)}_ia_{ij} = W^{(i)}_{jj}f^j,\quad i\neq j.
}
\ee{
\label{cond-p3}
P^{(i)}_ib_{ji} = W^{(i)}_{ij}f^j,\quad i\neq j.
}
We also compare the coefficients of $\qth_k^1\qth_j^2$, $\qth_i^1\qth_j^2$ and $\qth_j^1\qth_j^2$ in $M^i$ and $S^i$ to arrive at
\ee{\label{cond-p4-1}X^{(i)}_{kj} = P^{(i)}_{kj}u^j,\quad i\neq j\neq k\neq i.}
\ee{\label{cond-p4-2}
3P^{(i)}_i(u^i-u^j)b_{ji} = (X^{(i)}_{ij}-P^{(i)}_{ij}u^j)f^j,\quad i\neq j.
}
\ee{\label{cond-p4-3}
P^{(i)}_i(u^i-u^j)a_{ij} = (X^{(i)}_{jj}-P^{(i)}_{jj}u^j)f^j,\quad i\neq j.
}
By comparing the coefficients of $u^{k,1}\qth_j^2$, $u^{j,1}\qth_j^2$, $u^{i,1}\qth_j^2$, $u^{j,1}\qth_i^2$ and $u^{i,1}\qth_i^2$ in $N^i$ and $T^i$, and by using the relation \eqref{cond-p8-3} we arrive at the following identities:
\ee{\label{cond-p7-1}
Y^{(i)}_{jk} = Q^{(i)}_{jk}u^j,\quad i\neq j\neq k\neq i.
}
\ee{\label{cond-p7-2}
2Y^{(i)}_{jj}f^j+X^{(j)}_jb_{ij} = 2Q^{(i)}_{jj}u^jf^j+P^{(j)}_ju^ib_{ij},\quad i\neq j.
}
\ee{\label{cond-p7-3}
Y^{(i)}_{ji} = Q^{(i)}_{ji}u^j,\quad i\neq j.
}
\ee{\label{cond-p8-1}
Y^{(i)}_{ij} = Q^{(i)}_{ij}u^i,\quad i\neq j.
}
\ee{\label{cond-p8-2}
2Y^{(i)}_{ii} = 2Q^{(i)}_{ii}u^i+\frac 12 P^{(i)}_if^i.
}
\begin{proof}[\textbf{Proof of Lemma \ref{s3-t8} (Part 2)}] We need to show that a cocycle with vanishing indices must be a coboundary. Take a class $[\qo]\in\mathcal Z/\mathcal B$ with $\qo$ being its normal form and $ind_i(\qo) = 0$. Let $\qa\in\bar\Qo^1_2$ be the unique element given by \eqref{norm-5}--\eqref{norm-8} such that $\tilde D_0\qo = \tilde D_1\qa$. We will prove that $\qa=\qo=0$.

From $ind_i(\qo) = 0$ and \eqref{cond-p1-1} we know that $X^{(i)}_i = P^{(i)}_i = 0$. Then the identities \eqref{cond-p1-2} and \eqref{cond-p2-1}--\eqref{cond-p3} show that the coefficients $W^{(i)}_{jk}$ vanish when $k\ne j$ and $Z^{(i)}_{ji} = u^iW^{(i)}_{ji}$; the identities \eqref{cond-p4-1}--\eqref{cond-p4-3} together with \eqref{cond-p8-3} show that the coefficients $P^{(i)}_{jk}$ vanish when $k\ne i$ and $X^{(i)}_{ji} = u^iP^{(i)}_{ji}$. Similarly, we conclude from the identities \eqref{norm-4}, \eqref{norm-8} and \eqref{cond-p7-1}--\eqref{cond-p8-2} that the coefficients $Q^{(i)}_{jk}=0$ when $k\ne j$ and $Y^{(i)}_{jj} = u^jQ^{(i)}_{jj}$. In particular, from \eqref{norm-4} we know that $Q^{(i)}_{ii} = 0$.

To simplify the notations, we use $Y^{(i)}_j$ and $Q^{(i)}_j$ to denote $Y^{(i)}_{jj}$ and $Q^{(i)}_{jj}$. Therefore we can represent the coefficients of $\qo$ and $\qa$ that are given in \eqref{norm-2}, \eqref{norm-3} and \eqref{norm-6}, \eqref{norm-7} as follows:
\begin{align*}
X^i &=  \sum_{j}\left(X^{(i)}_{ji}u^{i,1}\qth_j^1+Z^{(i)}_{ji}u^{i,2}\qth_j\right)+\sum_{j,k,l}Z^{(i)}_{l;jk}u^{k,1}u^{j,1}\qth_l,\\
P^i &=  \sum_{j}\left(P^{(i)}_{ji}u^{i,1}\qth_j^1+W^{(i)}_{ji}u^{i,2}\qth_j\right)+\sum_{j,k,l}W^{(i)}_{l;jk}u^{k,1}u^{j,1}\qth_l,\\
Y^i &= \sum_jY^{(i)}_j(u^{j,1})^2,\quad Q^i = \sum_jQ^{(i)}_j(u^{j,1})^2,
\end{align*}
where the coefficients satisfy the conditions
\begin{align}
\label{pc-1}
&X^{(i)}_{ji} = u^iP^{(i)}_{ji},\quad Z^{(i)}_{ji} = u^iW^{(i)}_{ji};\\
\label{pc-2}
&Y^{(i)}_{j} = u^jQ^{(i)}_{j},\quad Y^{(i)}_i = Q^{(i)}_{i} = 0.
\end{align}

Let us first compare the coefficients of $u^{i,1}u^{i,2}\qth_j$ in $N^i$ and $T^i$ to arrive at
\ee{\label{np4}
2Z^{(i)}_{j;ii} = 2u^iW^{(i)}_{j;ii}-\frac 12 W^{(i)}_{ji},\quad\forall\ i,j.
}
At the same time we compute the coefficients of $u^{i,1}\qth_j\qth_i^2$ in $M^i$ and $S^i$, and we obtain
\ee{\label{np7-0}-2Z^{(i)}_{ji}\qp_if^i-2Z^{(i)}_{j;ii}f^i = -2W^{(i)}_{ji}\qp_i(u^if^i)-2W^{(i)}_{l;ii}u^if^i-\frac 12 W^{(i)}_{ji}f^i,\quad\forall\ i,j.}
Then it follows from the identities \eqref{pc-1}, \eqref{np4} and \eqref{np7-0} that 
\ee{\label{np7}Z^{(i)}_{ji} = W^{(i)}_{ji} = 0,\quad\forall\ i,j.}

Next we compare the coefficients of $u^{k,1}u^{j,2}\qth_l$ in $N^i$ and $T^i$ for distinct $j$, $k$ and we obtain the relations
\ee{\label{np2}
Z^{(i)}_{l;jk} = u^iW^{(i)}_{l;jk},\quad j\neq k,\quad\forall\ i,l.
}
For $j\neq i$, we compare the coefficients of $u^{j,1}\qth_l\qth_k^2$ in $M^i$ and $S^i$ and we arrive at
\ee{\label{np5-1}Z^{(i)}_{l;jk} = u^kW^{(i)}_{l;jk},\quad i\neq j,\quad\forall\ k,l.}
Since $Z^{i}_{l;jk} = Z^{(i)}_{l;kj}$, $W^{i}_{l;jk} = W^{(i)}_{l;kj}$, the identities \eqref{np2} and \eqref{np5-1} imply  that the only possibly non-vanishing coefficients are $Z^{(i)}_{l;jj}$ and $W^{(i)}_{l;jj}$; moreover, from the identities \eqref{np4} and \eqref{np7} we conclude that
\ee{\label{np5-2}Z^{(i)}_{l;jj} = u^jW^{(i)}_{l;jj},\quad\forall\ i,j,l.}

We proceed to compute the coefficients of $(u^{i,1})^2\qth_j^1$ in $N^i$ and $T^i$. By using \eqref{pc-1} and \eqref{pc-2}, it is easy to deduce that
\ee{\label{np14-2}X^{(i)}_{ji} = P^{(i)}_{ji} = 0,\quad\forall\ i,j.}
Thanks to this equation, we can also compute  the coefficients of $(u^{j,1})^2\qth_j^1$ of $N^i$ and $T^i$ and obtain the relations
\ee{\label{np14-0}f^iZ^{(i)}_{j;jj} = u^if^iW^{(i)}_{j;jj}+2Q^{(i)}_jf^j,\quad i\neq j. }
On the other hand, it follows from \eqref{np7} and the coefficients of $u^{j,1}u^{j,2}\qth_j$ in $N^i$ and $T^i$ for $i\neq j$ that
\ee{\label{np3-2}
2f^iZ^{(i)}_{j;jj} = 2u^if^iW^{(i)}_{j;jj}+Q^{(i)}_jf^j,\quad i\neq j.
}
Thus the relations \eqref{np14-0} and \eqref{np3-2} together with \eqref{pc-2} lead to
\ee{\label{np14-1}Y^i = Q^i = 0,\quad\forall\ i.}

Now thanks to \eqref{np7}, \eqref{np14-2} and \eqref{np14-1}, we are able to easily compare the coefficients of $u^{j,1}u^{j,2}\qth_k$ and $u^{j,1}u^{j,2}\qth_i$ in $N^i$ and $T^i$ respectively and we arrive at
\aaa{\label{np3-1}
Z^{(i)}_{k;jj}
&=u^iW^{(i)}_{k;jj},\quad i\neq j\neq k\neq i;\\
\label{np3-3}
Z^{(i)}_{i;jj}
&=u^iW^{(i)}_{i;jj},\quad i\neq j,
}
Therefore by combining the identities \eqref{np5-2}, \eqref{np3-1} and \eqref{np3-3} we conclude that
\[
X^i = \sum_jZ^{(i)}_{j;ii}(u^{i,1})^2\qth_j,\quad P^i = \sum_jW^{(i)}_{j;ii}(u^{i,1})^2\qth_j,\quad Z^{(i)}_{j;ii} = u^iW^{(i)}_{j;ii}.\]

and comparing the coefficients of $(u^{i,1})^3$ we deduce that $P^i=0$,  and it follows from the above relation that $X^i$ also vanish. The lemma is proved.

\end{proof}

\section{Vanishing theorem of the variational bihamiltonian cohomology}\label{vsh}
\subsection{Vanishing theorem and the strategy of computation.}
In this section, we compute the general variational bihamiltonian cohomology groups $\vbh^p_d(\bar\Qo,\tilde D_0,\tilde D_1)$. The main result we are to obtain is the following theorem.
\begin{Th}
\label{thm-van}
The cohomology group $H^p_d(\Qo[\ql],\qp_\ql)$ vanishes unless the bidegree $(p,d)$ belongs to the following two cases:
\begin{align*}
&\mathrm{Case\, 1:}\quad d = 0,\cdots,n;\quad p = d+1,\cdots,d+n+1,\\
&\mathrm{Case\, 2:}\quad d = 2,\cdots,n+3;\quad p = d,\cdots,d+n.
\end{align*}
\end{Th}
%\clr{This theorem, together Lemme \ref{lem-lambda}, Theorem \ref{h23} and Theorem \ref{vbh12} yields Theorem \ref{vbh-res}.}
\begin{proof}[Proof of Theorem \ref{vbh-res}]
Similar to Lemma \ref{s3-t3}, we have the following long exact sequence:
\begin{equation*}
\cdots\rightarrow H^p_d(\Qo[\ql],\qp_\ql)\rightarrow H^p_d(\bar\Qo[\ql],\qp_\ql)\rightarrow H^{p+1}_d(\Qo[\ql],\qp_\ql)\rightarrow\cdots.
\end{equation*}
Then Theorem \ref{vbh-res} follows from Theorem \ref{thm-van} together with Lemma \ref{lem-lambda},
Theorem \ref{h23} and Theorem \ref{vbh12}.
\end{proof}
%\begin{Cor}[Theorem \ref{vbh-res}]
%\clr{The variational bihamiltonian cohomology groups of the bihamiltonian structure $(P_0, P_1)$ have the following properties:
%\begin{align*}
%&(1)\ \vbh^1_2(\bar\Qo,\tilde D_0,\tilde D_1) = 0,\\
%&(2)\ \vbh^2_3(\bar\Qo,\tilde D_0,\tilde D_1)\cong\oplus_{i=1}^nC^\infty(\mathbb R),\\
%&(3)\ \vbh^p_d(\bar\Qo,\tilde D_0,\tilde D_1) = 0,\quad \textrm{for}\ d\ge 2,\,\textrm{and}\
%(p, d),\, (p+1, d)\notin I.
%\end{align*}
%Here $I = I_1\cup I_2\cup I_3$ with 
%\begin{align*}
%I_1 &= \{(x,y):\ y = 0,1;\ x = y+1,\cdots,y+n+1\};\\
%I_2 &= \{(x,y):\ y = 2,\cdots,n;\ x = y,\cdots,y+n+1\};\\
%I_3 &= \{(x,y):\ y = n+1,n+2,n+3;\ x = y,\cdots,y+n\}.
%\end{align*}}
%
%\begin{proof}
%Similarly as Lemma \ref{s3-t3}, we have the following long exact sequence:
%\begin{equation*}
%\cdots\rightarrow H^p_d(\Qo[\ql],\qp_\ql)\rightarrow H^p_d(\bar\Qo[\ql],\qp_\ql)\rightarrow H^{p+1}_d(\Qo[\ql],\qp_\ql)\rightarrow\cdots.
%\end{equation*}
%Using this exact sequence, this corollary is straightforward from Theorem \ref{thm-van}.
%\end{proof}

The strategy to prove the above theorem is inspired by the work \cite{carlet2018deformations,carlet2018central}, where some appropriate spectral sequences are used to compute the bihamiltonian cohomology of a bihamiltonian structure of hydrodynamic type. Our computation in this section can be viewed as a certain generalization of that of \cite{carlet2018central}, therefore we will use the same notations as the ones used in \cite{carlet2018central} whenever possible.

For a semisimple bihamiltonian structure $(P_0,P_1)$ of hydrodynamic type, we will work in its canonical coordinates $u^1,\cdots, u^n$. Given $\qo\in\Qo^p$, we represent it in the form
\[\qo = \sum_{i;s\geq 0}g_{i,s}\qd u^{i,s}+h^i_s\qd\qth_i^s,\quad g_{i,s}\in\hm A^p,\ h^i_s\in\hm A^{p-1},\]
then by using the formula given in Example \ref{ex-3-2-zh} we have
\[\tilde D_a\qo = \sum_{i;s\geq 0}(D_{P_a}g_{i,s})\qd u^{i,s}+(-1)^pg_{i,s}\qd (D_{P_a}u^{i,s})+(D_{P_a}h^i_s)\qd\qth_i^s+(-1)^{p-1}h^i_s\qd (D_{P_a}\qth_i^s)\]
for $a=0, 1$. 
We are going to compute the cohomology of the complex $(\Qo[\ql],\qp_\ql=\tilde D_1-\ql\tilde D_0)$ in what follows.

Let us define a gradation, called the $u$-degree, on $\Qo$ by setting 
\begin{equation*}
\deg_u u^{i,s+1} = 1,\ \deg_u \qd u^{i,s} = 1,\quad i=1,\cdots,n;\quad s\geq 0,
\end{equation*}
and the $u$-degrees of other generators are set to be 0. Let us denote the super degree of an element $\qo$ of $\Qo$ by $\deg_\qth\qo$. We filtrate the complex $\Qo[\ql]$ by defining
\[
F^k\Qo[\ql] = \{\qo\in\Qo[\ql]\mid\deg_u\qo+\deg_\qth\qo\geq k\},
\]
then we have a filtration
\[\cdots F^{k+1}\Qo[\ql]\subset F^k\Qo[\ql]\subset\cdots F^0\Qo[\ql] = \Qo[\ql].\]
We also decompose the differential in the following way:
\[\qp_\ql = \Qd_{-1}+\Qd_0+\cdots,\quad \deg_u\Qd_k = k.\]
Note that $\deg_\qth\qp_\ql = 1$, therefore each $\Qd_k$ preserves the filtration. Now let $(^1E_r,d_r), r\geq 0$ be the spectral sequence induced by the filtration, then by the standard construction $(^1E_0,d_0) = (\Qo[\ql],\Qd_{-1})$, and $^1E_k$ is given by the cohomology of $^1E_{k-1}$ for $k\geq 1$. It is clear that when restricted to $\Qo^p_d[\ql]$, the above filtration is bounded and therefore this guarantees the convergence of the spectral sequence. Note that the differential $d_r$ with $r=0$ or $1$ is not the
induced differential on $\hm F$ used in Sect.\,\ref{vbc}.

We will compute $^1E_1$ and $^1E_2$ in the following subsections, to this end we need to construct some other spectral sequences as in \cite{carlet2018central}. The following standard fact is also used in the computation.
\begin{Lem}
\label{ha-lem}
Let $(\mathcal C^\bullet,d)$ be a cochain complex in an Abelian category. Assume that for each $p$ we have the decomposition $\mathcal C^p = \mathcal A^p\oplus \mathcal B^p$ with $\mathcal A^\bullet$ being a subcomplex, and $\mathcal A^\bullet$ is acyclic. Then 
\begin{equation}
H^p(\mathcal C^\bullet,d)\cong H^p(\mathcal B^\bullet,\pi_B\circ d).
\end{equation}
Here $\pi_B$ is the projection from $\mathcal C^\bullet$ to $\mathcal B^\bullet$.
\end{Lem}

\subsection{Computation of $^1E_1$.}\label{subsec-5-2-zh} We first write down $\Qd_{-1}$ in the explicit form
\begin{align}
\label{Qd-1}
\Qd_{-1} &= \sum_{s\geq 1;i}(-\ql+u^i)f^i\qth_i^{s+1}\diff{}{u^{i,s}}+\sum_{s\geq 0;i}(-\ql+u^i)f^i\qd\qth_i^{s+1}\diff{}{\qd u^{i,s}}\\
\nonumber
&:=\sum_i(-\ql+u^i)f^i\qd\hat d_i,
\end{align}
here we introduce the de Rham-type differential
\begin{equation}
\label{delta-di}
\qd\hat d_i = \sum_{s\geq 1}\qth_i^{s+1}\diff{}{u^{i,s}}+\sum_{s\geq 0}\qd\qth_i^{s+1}\diff{}{\qd u^{i,s}}.
\end{equation}
We will also use the notation
\begin{equation}
\label{di}
\hat d_i = \sum_{s\geq 1}\qth_i^{s+1}\diff{}{u^{i,s}}.
\end{equation}
\begin{Lem}
Each summand of the following decomposition is a cochain subcomplex with respect to $\Delta_{-1}$:
\begin{equation}
\label{decomp1}
\Qo[\ql] = \hm{A}[\ql]\{\qd\qth_i\mid i=1,\cdots,n\}\oplus\bigoplus_{i=1}^n\hm{A}[\ql]\{\qd u^{i,s},\qd\qth_i^{s+1}\mid s\geq 0\}.
\end{equation}
Here we use the notation $R\{g_1, g_2, \dots\}$ to denote the free $R$-module generated by $g_1, g_2, \dots$.
\end{Lem}

In what follows, we will use the same notations as in \cite{carlet2018deformations} to denote the following subspaces:
\begin{align*}
\hm{C} &= C^\infty(u)[\qth_1,\cdots,\qth_n,\qth_1^1,\cdots,\qth_n^1],\\
\hm{C}_i &= \hm{C}[[ u^{i,s},\qth_i^{s+1}\mid s\geq 1]],
\end{align*}
and use $\hm{C}_i^{nt}$ to denote the subspace of $\hm{C}_i$ spanned by nontrivial monomials, i.e., all monomials that contain at least one of the variables $u^{i,s}$, $\qth_i^{s+1}$ for $s\geq 1$. We also use $\hm{M}$ to denote the subspace of $\hm A$ spanned by monomials which contain at least one of the mixed quadratic expressions:
\begin{equation*}
u^{i,s}u^{j,t};\quad u^{i,s}\qth_j^{t+1};\quad \qth_i^{s+1}\qth_j^{t+1}
\end{equation*}
for some $i\neq j$ and $s,t\geq 1$. Then it is easy to see that there is a decomposition
\begin{equation}
\label{decomp2}
\hm A = \hm C\oplus\kk{\bigoplus_{i=1}^n\hm C_i^{nt}}\oplus\hm M,
\end{equation}
and each summand is preserved under the action of $\hat d_i$ given by \eqref{di}.
\begin{Lem}
We have
\begin{equation*}
H(\hm{A}[\ql]\{\qd\qth\},\Qd_{-1}) = \hm{C}[\ql]\{\qd\qth\}\oplus\bigoplus_{i=1}^n\frac{\hat d_i(\hm C_i[\ql])}{(-\ql+u^i)\hat d_i(\hm C_i[\ql])}\{\qd\qth\}.
\end{equation*}
Here $\{\qd\qth\}$ is the abbreviation of $\{\qd\qth_i,\ i=1,\cdots,n\}$.
\end{Lem}
\begin{proof}
It is clear from the formula \eqref{Qd-1} that
\begin{equation*}
H(\hm{A}[\ql]\{\qd\qth\},\Qd_{-1}) = H(\hm{A}[\ql],\Qd_{-1})\{\qd\qth\},
\end{equation*}
where the action of $\Qd_{-1}$ on $\hm A[\ql]$ is just $\sum_i (-\ql+u^i)f^i\hat d_i$. Hence the lemma is proved by applying the result in \cite{carlet2018deformations}.
\end{proof}

According to the decomposition \eqref{decomp1}, we still need to compute cohomology
\[H(\hm{A}[\ql]\{\qd u^{i,s},\qd\qth_i^{s+1}\mid s\geq 0\},\Qd_{-1}).\]
To this end, we construct a second spectral sequence $^2E$. For a fixed index $i$, we define the $u^i$-degree by setting
\begin{equation*}
\deg_{u^i}u^{i,s+1}= 1,\quad \deg_{u^i}\qd u^{i,s} = 1,\quad s\geq 0,
\end{equation*}
and other generators have $u^i$-degree zero. Accordingly, we decompose the differential $\Qd_{-1}$ as follows
\begin{equation*}
\Qd_{-1}=\Qd_{-1,-1}+\Qd_{-1,0},\quad \deg_{u^i}\Qd_{-1,k} = k,
\end{equation*}
where
\[
\Qd_{-1,-1} = (-\ql+u^i)f^i\qd\hat d_i,\quad \Qd_{-1,0} = \sum_{j\neq i}(-\ql+u^j)f^j\qd\hat d_j.
\]
Similar to our construction of $^1E$, we filtrate $\hm{A}[\ql]\{\qd u^{i,s},\qd\qth_i^{s+1}\mid s\geq 0\}$ with $\deg_{u^i}+\deg_{\qth}$, and construct the associated spectral sequence $^2E$. Then we have
\[(^2E_0,d_0) = (\hm{A}[\ql]\{\qd u^{i,s},\qd\qth_i^{s+1}\mid s\geq 0\},\Qd_{-1,-1})\]
and $^2E_1 = H(^2E_0,d_0)$, with differential $d_1 = \Qd_{-1,0}$. Since on the $^2E_2$ page the differential becomes 0, hence this spectral sequence becomes convergent on this page, i.e.
\[H(\hm{A}[\ql]\{\qd u^{i,s},\qd\qth_i^{s+1}\mid s\geq 0\}, \Qd_{-1}) \cong{^2E_2}.\]

Let us compute $^2E_1 = H(^2E_0,d_0)$. Take any element
\begin{equation}
\label{s5-t1}
\qo = \sum_{s\geq 0}p_s\qd u^{i,s}+q_s\qd\qth_i^{s+1}\in {^2E_0},
\end{equation}
where we assume that $\qo$ is homogeneous of degree $p$ with respect to the $\qth$-degree. Then it is easy to see that
\begin{equation*}
d_0\qo = (-\ql+u^i)f^i\sum_{s\geq 0}\hat d_i(p_s)\qd u^{i,s}+\left(\hat d_i(q_s)+(-1)^pp_s\right)\qd\qth_i^{s+1}.
\end{equation*}
So if $d_0\qo = 0$, we have $p_s = (-1)^{p+1}\hat d_i(q_s)$, and we can write
\begin{equation*}
\qo = \qd\hat d_i\biggl(\sum_{s\geq 0}(-1)^{p+1}q_s\qd u^{i,s}\biggr).
\end{equation*}
Let us denote \[\hm B_i = \hm A\{\qd u^{i,s},\ s\geq 0\},\] then the above computation shows that $\ker d_0\subset \qd\hat d_i(\hm B_i)[\ql]$. The inverse inclusion $\qd\hat d_i(\hm B_i)[\ql]\subset \ker d_0$ is clear, hence we have $\ker d_0=\qd\hat d_i(\hm B_i)[\ql]$. 

To compute the image of the differential $d_0$, we take another element $\qo\in{^2E_0}$ represented in the form of \eqref{s5-t1}. Observe that
\begin{align*}
d_0\qo &= (-\ql+u^i)f^i\sum_{s\geq 0}\hat d_i(p_s)\qd u^{i,s}+\left(\hat d_i(q_s)+(-1)^pp_s\right)\qd\qth_i^{s+1}\\
&=(-\ql+u^i)f^i\qd\hat d_i\sum_{s\geq 0}\left(p_s+(-1)^p\hat d_i(q_s)\right)\qd u^{i,s},
\end{align*}
we see that $\ima d_0 = (-\ql+u^i)\qd\hat d_i(\hm B_i)[\ql]$. Therefore we obtain the following lemma.
\begin{Lem} We have
\begin{equation*}
^2E_1 = \frac{\qd\hat d_i(\hm B_i)[\ql]}{(-\ql+u^i)\qd\hat d_i(\hm B_i)[\ql]}.
\end{equation*}
\end{Lem}

Next let us compute $^2E_2 = H(^2E_1,d_1)$. We use the following identification:
\begin{equation*}
^2E_1 = \frac{\qd\hat d_i(\hm B_i)[\ql]}{(-\ql+u^i)\qd\hat d_i(\hm B_i)[\ql]}\cong \qd\hat d_i(\hm B_i)
\end{equation*}
by identifying $\ql$ with $u^i$. Under this identification, we have
\begin{equation*}
d_1 = \sum_j(u^j-u^i)f^j\qd\hat d_j.
\end{equation*}
Thus we can represent an element of $^2E_1$ as follows:
\begin{equation*}
\qo = \qd\hat d_i \sum_{s\geq 0}p_s\qd u^{i,s} = \sum_{s\geq 0}\hat d_i(p_s)\qd u^{i,s}+(-1)^pp_s\qd\qth_i^{s+1}\in {^2E_1}.
\end{equation*}
If $\qo$ is a cocycle, i.e. $d_1\qo = 0$, by using the fact that $\qd\hat d_i$ commutes with $d_1$ we obtain
\begin{equation}
\label{s5-t2}
\sum_{j}(u^j-u^i)f^j\hat d_j(p_s) = 0.
\end{equation}
Let us consider the following decomposition according to \eqref{decomp2} as follows:
\begin{equation*}
p_s = \hat p_s+\sum_k  p_s^k+p_s^m,
\end{equation*}
where $\hat p_s \in \hm C$, $ p_s^j \in \hm C_j^{nt}$ and $p_s^m\in\hm M$. Then from the equation \eqref{s5-t2} and the obvious fact that $\hat d_j$ annihilates $p_s^k$ unless $k=j$, we arrive at the following equations:
\begin{equation}
\left(u^k-u^i\right)f^k\hat d_k(p_s^k) = 0,\quad \sum_{j}(u^j-u^i)f^j\hat d_j(p_s^m) = 0.\label{ukuifk}
\end{equation}
Therefore for $k\neq i$, we have $\hat d_k(p_s^k) = 0$ and by Poincar\'e Lemma (Lemma 12 in \cite{carlet2018deformations}), this implies $p_s^k = \hat d_k(q_s^k)$ for some $q_s^k\in\hm C_k^{nt}$. For the term $p_s^m\in\hm M$, after a rescaling by a non-zero factor, we may rewrite the second equation of \eqref{ukuifk} as $\sum_{j\neq i}\hat d_j(p_s^m) = 0$, then by Proposition 11 in \cite{carlet2018deformations} (or a simplified version, Lemma 3.4 in \cite{carlet2018central}), we see that $p_s^m = \sum_{j\neq i}\hat d_j(q_s^m)+h_s$, where $q_s^m\in\hm M$ and $h_s\in\hm C_i$. But such $h_s$ must vanish since we have $h_s = p_s^m - \sum_{j\neq i}\hat d_j(q_s^m)\in\hm M$.

To conclude, we can rewrite $p_s$ as
\begin{equation*}
p_s = \hat p_s+p_s^i+\sum_{j\neq i}\hat d_i q_s^j+d_1(q_s^m) = \hat p_s+p_s^i+ d_1\left(\sum_{j\neq i}\frac{q_s^j}{(u^j-u^i)f^j}+q_s^m\right).
\end{equation*}Using again the fact that $\qd\hat d_i$ and $d_1$ commutes, we see that
\begin{align*}
\qo &= \qd\hat d_i (\sum_{s\geq 0}p_s\qd u^{i,s})\\
&= \qd\hat d_i\sum_{s\geq 0}(\hat p_s+p_s^i)\qd u^{i,s}-d_1\qd\hat d_i\sum_{s\geq 0}\sum_{j\neq i}\left(\frac{q_s^j}{(u^j-u^i)f^j}+q_s^m\right)\qd u^{i,s}.
\end{align*}
So we see that for any $\qo \in\ker d_1$, there is $\eta \in \hm C_i\{\qd u^{i,s}\mid s\geq 0\}$ such that $[\qo] = [\qd\hat d_i(\eta)]$ in $^2E_2$. Let us denote $\hm H_i =\hm C_i\{\qd u^{i,s},\ s\geq 0\}$. It is obvious that each element in $\qd\hat d_i(\hm H_i)$ is annihilated by $d_1$ and different elements define different classes in $^2E_2$, hence we arrive at
\begin{equation*}
^2E_2 \cong \qd\hat d_i(\hm H_i).
\end{equation*}

To summarize all the results above, we obtain the following theorem.
\begin{Th}\label{thm-5-6-zh}
The first page $^1E_1$ of the spectral sequence $^1E$ can be described as the direct sum of the following spaces:
\begin{equation*}
^1E_1\cong \hm{C}[\ql]\{\qd\qth\}\oplus\bigoplus_{i=1}^n\frac{\hat d_i(\hm C_i[\ql])}{(-\ql+u^i)\hat d_i(\hm C_i[\ql])}\{\qd\qth\}\oplus\bigoplus_{i=1}^n\qd\hat d_i(\hm H_i).
\end{equation*}
\end{Th}
\subsection{Computation of $^1E_2$.}
We will find suitable bidegrees $(p,d)$ such that the cohomology $^1E_2 = H^p_d(^1E_1,\Qd_0) = 0$, this means that the spectral sequence $^1E$ collapses on the second page for these bidegrees. Then we conclude that for the same bidegrees $H^p_d(\Qo[\ql],\qp_\ql) = 0$. 

We first write down explicitly the formula for $\Qd_0$. To avoid lengthy expressions, we will split $\Qd_0$ into $\qp/\qp u^{i,s}$ part, $\qp/\qp \qd u^{i,s}$ part, $\qp/\qp \qth_i^s$ part and $\qp/\qp \qd\qth_i^s$ part. In the following formulae, the index $i$ is fixed and does not participate in the summation. These formulae are comparable to those given in \cite{carlet2018deformations}.

The $\qp/\qp u^{i,s}$ part of $\Qd_0$ reads
\begin{align*}
&\sum_{s\geq 1}\sum_{t=1}^s\binom{s}{t}u^{i,t}f^i\qth_i^{s-t+1}\diff{}{u^{i,s}}
+\sum_{s\geq 1;j}\sum_{t=1}^s\binom{s}{t}(-\ql+u^i)\aij ji u^{j,t}\qth_i^{s-t+1}\diff{}{u^{i,s}}\\
+&\frac 12\sum_{s\geq 1}\sum_{t=0}^s\binom{s}{t}u^{i,t+1}f^i\qth_i^{s-t}\diff{}{u^{i,s}}
+\frac 12\sum_{s\geq 1;j}\sum_{t=0}^s\binom{s}{t}(-\ql+u^i)\aij ji u^{j,t+1}\qth_i^{s-t}\diff{}{u^{i,s}}\\
+&\frac 12\sum_{s\geq 1;j}\sum_{t=0}^s\binom{s}{t}(-\ql+u^i)\bij ij u^{j,t+1}\qth_j^{s-t}\diff{}{u^{i,s}}\\
-&\frac 12 \sum_{s\geq 1;j}\sum_{t=0}^s\binom{s}{t}(-\ql+u^j)\bij ji u^{i,t+1}\qth_j^{s-t}\diff{}{u^{i,s}}+(\ql+u^i)f^i\qth_i^1\diff{}{u^i}.
\end{align*}
The $\qp/\qp \qd u^{i,s}$ part part of $\Qd_0$ reads:
\begin{align*}
&\sum_{s\geq 0}\sum_{t=0}^s\binom{s}{t}f^i\qth_i^{s-t+1}\qd u^{i,t}\diff{}{\qd u^{i,s}}
+\sum_{s\geq 0;j}\sum_{t=0}^s\binom{s}{t}(-\ql+u^i)\aij ji \qth_i^{s-t+1}\qd u^{j,t}\diff{}{\qd u^{i,s}}\\
+&\sum_{s\geq 1}\sum_{t=1}^s\binom{s}{t}f^iu^{i,t}\qd\qth_i^{s-t+1}\diff{}{\qd u^{i,s}}
+\sum_{s\geq 1;j}\sum_{t=1}^s\binom{s}{t}(-\ql+u^i)\aij ji u^{j,t}\qd\qth_i^{s-t+1}\diff{}{\qd u^{i,s}}\\
+&\frac 12\sum_{s\geq 0}\sum_{t=0}^s\binom{s}{t}f^i\qd(u^{i,t+1}\qth_i^{s-t})\diff{}{\qd u^{i,s}}
+\frac 12\sum_{s\geq 0;j}\sum_{t=0}^s\binom{s}{t}(-\ql+u^i)\aij ji \qd(u^{j,t+1}\qth_i^{s-t})\diff{}{\qd u^{i,s}}\\
+&\frac 12\sum_{s\geq 1;j}\sum_{t=0}^s\binom{s}{t}(-\ql+u^i)\bij ij \qd(u^{j,t+1}\qth_j^{s-t})\diff{}{\qd u^{i,s}}\\
-&\frac 12 \sum_{s\geq 1;j}\sum_{t=0}^s\binom{s}{t}(-\ql+u^j)\bij ji \qd(u^{i,t+1}\qth_j^{s-t})\diff{}{\qd u^{i,s}}.
\end{align*}
The $\qp/\qp \qth_i^s$ part part of $\Qd_0$ reads
\begin{align*}
&\frac 12 \sum_{s\geq 0}\sum_{t=0}^s\binom{s}{t}f^i\qth_i^t\qth_i^{1+s-t}\diff{}{\qth_i^s}+\frac 12\sum_{s\geq 0;j}\sum_{t=0}^s\binom{s}{t}(-\ql+u^j)\aij ij \qth_j^t\qth_j^{1+s-t}\diff{}{\qth_i^s}\\
+&\frac 12 \sum_{s\geq 0;j}\sum_{t=0}^s\binom{s}{t}(-\ql+u^j)\bij ji \qth_i^t\qth_j^{1+s-t}\diff{}{\qth_i^s}\\
-&\frac 12\sum_{s\geq 0;j}\sum_{t=0}^s\binom{s}{t}(-\ql+u^j)\bij ji \qth_j^t\qth_i^{1+s-t}\diff{}{\qth_i^s}.
\end{align*}
The $\qp/\qp \qd\qth_i^s$ part part of $\Qd_0$ reads
\begin{align*}
&\frac 12 \sum_{s\geq 0}\sum_{t=0}^s\binom{s}{t}f^i\qd(\qth_i^t\qth_i^{1+s-t})\diff{}{\qd\qth_i^s}+\frac 12\sum_{s\geq 0;j}\sum_{t=0}^s\binom{s}{t}(-\ql+u^j)\aij ij \qd(\qth_j^t\qth_j^{1+s-t})\diff{}{\qd\qth_i^s}\\
+&\frac 12 \sum_{s\geq 0;j}\sum_{t=0}^s\binom{s}{t}(-\ql+u^j)\bij ji \qd(\qth_i^t\qth_j^{1+s-t})\diff{}{\qd\qth_i^s}\\
-&\frac 12\sum_{s\geq 0;j}\sum_{t=0}^s\binom{s}{t}(-\ql+u^j)\bij ji \qd(\qth_j^t\qth_i^{1+s-t})\diff{}{\qd\qth_i^s}.
\end{align*}

To compute the cohomology $H({^1E_1},\Qd_0)$, we introduce a third spectral sequence $^3E$ by defining the $\qth^1$-degree
$\deg_{\qth^1}\qth_i^1 = 1$
for $i=1,\cdots, n$, and other generators have $\deg_{\qth^1}$ zero. Then we filtrate ${^1E_1}$ via
\[F^r({^1E_1}) = \{\deg_{\qth_1}-\deg_\qth \qo\leq -r\mid\qo\in {^1E_1}\}.\]
We also have the decomposition
\begin{equation*}
\Qd_0 = \Qd_{0,1}+\Qd_{0,0}+\Qd_{0,-1},\quad \deg_{\qth^1}\Qd_{0,k} = k.
\end{equation*}
The zeroth page  $(^3E_0,d_0)$ is given by $(^1E_1,\Qd_{0,1})$. We want to determine the bidegrees $(p,d)$ such that $H^p_d(^3E_0,d_0) = 0$, then we conclude that for the same degrees $H^p_d(^1E_1,\Qd_0) = 0$. We first write down the explicit formula for $\Qd_{0,1}$.

The $\qp/\qp u^{i,s}$ part of $\Qd_{0,1}$ reads
\begin{align*}
&(-\ql+u^i)f^i\qth_i^1\diff{}{u^i}+\sum_{s\geq 1}(\frac s2+1)f^i\qth_i^1u^{i,s}\diff{}{u^{i,s}}+\sum_{s\geq 1;j}(\frac s2+1)(-\ql+u^i)\aij ji \qth_i^1u^{j,s}\diff{}{u^{i,s}}\\
+&\frac 12 \sum_{s\geq 1;j}s(-\ql+u^i)\bij ij \qth_j^1u^{j,s}\diff{}{u^{i,s}}-\frac 12 \sum_{s\geq 1;j}s(-\ql+u^j)\bij ji \qth_j^1u^{i,s}\diff{}{u^{i,s}}.
\end{align*}
The $\qp/\qp \qd u^{i,s}$ part of $\Qd_{0,1}$ reads
\begin{align*}
&\sum_{s\geq 0}f^i\qth_i^1\qd u^{i,s}\diff{}{\qd u^{i,s}}+\sum_{s\geq 0;j}(-\ql+u^i)\aij ji \qth_i^1\qd u^{j,s}\diff{}{u^{i,s}}\\
+&\frac 12\sum_{s\geq 0;j}s(-\ql+u^i)\aij ji \qth_i^1\qd u^{j,s}\diff{}{\qd u^{i,s}}+\frac 12 \sum_{s\geq 0}sf^i\qth_i^1\qd u^{i,s}\diff{}{\qd u^{i,s}}\\
+&\frac 12 \sum_{s\geq 0;j}s(-\ql+u^i)\bij ij \qth_j^1\qd u^{j,s}\diff{}{\qd u^{i,s}}-\frac 12 \sum_{s\geq 0;j}s(-\ql+u^j)\bij ji\qth_j^1\qd u^{i,s}\diff{}{\qd u^{i,s}}.
\end{align*}
The $\qp/\qp \qth_i^s$ part part of $\Qd_{0,1}$
\begin{align*}
&\frac 12\sum_{s\geq 0;j}(-\ql+u^j)\aij ij (s-1)\qth_j^1\qth_j^s\diff{}{\qth_i^s}+\frac 12\sum_{s\geq 0}f^i(s-1)\qth_i^1\qth_i^s\diff{}{\qth_i^s}\\
+&\frac 12 \sum_{\substack{s\geq 0;j\\ s\neq 1}}(\ql+u^j)\bij ji (s+1)(\qth_i^s\qth_j^1-\qth_j^s\qth_i^1)\diff{}{\qth_i^s}+\sum_j(-\ql+u^j)\bij ji \qth_i^1\qth_j^1\diff{}{\qth_i^1}
\end{align*}
The $\qp/\qp \qd\qth_i^s$ part part of $\Qd_0$
\begin{align*}
&\frac 12\sum_{s\geq 0;j}(-\ql+u^j)\aij ij (s-1)\qth_j^1\qd\qth_j^s\diff{}{\qd\qth_i^s}+\frac 12\sum_{s\geq 0}f^i(s-1)\qth_i^1\qd\qth_i^s\diff{}{\qth_i^s}\\
+&\frac 12 \sum_{\substack{s\geq 0;j\\ s\neq 1}}(\ql+u^j)\bij ji (s+1)(\qd\qth_i^s\qth_j^1-\qd\qth_j^s\qth_i^1)\diff{}{\qth_i^s}+\sum_j(-\ql+u^j)\bij ji \qd(\qth_i^1\qth_j^1)\diff{}{\qth_i^1}
\end{align*}

To simplify the above expressions, we perform a rescaling on the generators of $\Qo$ (see \cite{carlet2018central}) as follows:
\begin{align*}
&\Psi\colon u^{i,s}\mapsto (f^i)^{\frac s2}u^{i,s};\quad \qth_i^s\mapsto (f^i)^{\frac{s+1}{2}}\qth_i^s;\ s\geq 0;\\
&\Psi\colon \qd u^{i,s}\mapsto (f^i)^{\frac s2}\qd u^{i,s};\quad \qd\qth_i^s\mapsto (f^i)^{\frac{s+1}{2}}\qd\qth_i^s;\ s\geq 0.
\end{align*}
Note that this is NOT induced by a change of coordinate, but just an isomorphism of the space $\Qo$. The expression of $\Qd_{0,1}$ will be simplified after being conjugated by $\Psi$, and since $\Psi$ leaves all the decomposition of the complex invariant, this will not affect the computation of cohomology groups.

The following identities are easy to verify and helpful for our computation of the conjugated operator $\tilde\Qd_{0,1}=\Psi^{-1}\Qd_{0,1}\Psi$:
\begin{align*}
&\Psi^{-1}u^{i,s}\Psi = (f^i)^{-\frac s2}u^{i,s};\quad \Psi^{-1}\qth_i^s\Psi = (f^i)^{-\frac{s+1}{2}}\qth_i^s;\\
&\Psi^{-1}\qd u^{i,s}\Psi = (f^i)^{-\frac s2}u^{i,s};\quad \Psi^{-1}\qd\qth_i^s\Psi = (f^i)^{-\frac{s+1}{2}}\qth_i^s;\\
&\Psi^{-1}\diff{}{u^{i,s}}\Psi = (f^i)^{\frac s2}\diff{}{u^{i,s}}, s\geq 1;\quad \Psi^{-1}\diff{}{\qth_i^s}\Psi = (f^i)^{\frac{s+1}{2}}\diff{}{\qth_i^s}, s\geq 0;\\
&\Psi^{-1}\diff{}{\qd u^{i,s}}\Psi = (f^i)^{\frac s2}\diff{}{\qd u^{i,s}}, s\geq 0;\quad \Psi^{-1}\diff{}{\qd\qth_i^s}\Psi = (f^i)^{\frac{s+1}{2}}\diff{}{\qd\qth_i^s}, s\geq 0;\\
&\Psi^{-1}\diff{}{u^{i}}\Psi =\diff{}{u^i}+\sum_{s\geq 0;j}\frac{\aij ij}{f^j}\left(\frac s2 u^{j,s}\diff{}{u^{j,s}}+\frac{s+1}{2}\qth_j^s\diff{}{\qth_j^s}\right).
\end{align*}
Then by using the rotation coefficients 
\begin{equation*}
\qg_{ij} = -\frac 12 \left(\frac{f^i}{f^j}\right)^{1/2}\frac{\aij ij}{f^j}
\end{equation*}
defined for the diagonal metric $(f^1,\cdots,f^n)$,
we can represent $\tilde \Qd_{0,1}$ in the form
\[ \tilde \Qd_{0,1}= \phi_1+\phi_2+\phi_3,\]
where
\begin{align*}
\phi_1 =& -\sum_{s\geq 1;i,j}(-\ql+u^i)\left(\frac{f^i}{f^j}\right)^{\frac{s+1}{2}}\left((s+2)\qg_{ji}\qth_i^1+s\qg_{ij}\qth_j^1\right)u^{j,s}\diff{}{u^{i,s}}\\
&-\sum_{s\geq 0;i,j}(-\ql+u^i)\left(\frac{f^i}{f^j}\right)^{\frac{s+1}{2}}\left((s+2)\qg_{ji}\qth_i^1+s\qg_{ij}\qth_j^1\right)\qd u^{j,s}\diff{}{\qd u^{i,s}}\\
&+\sum_{s\geq 2;i,j}(-\ql+u^j)\left(\frac{f^i}{f^j}\right)^{\frac{s}{2}}\left((1-s)\qg_{ij}\qth_j^1-(1+s)\qg_{ji}\qth_i^1\right)\qth_j^s\diff{}{\qth_i^s},\\
\phi_2=& \sum_{s\geq 2;i,j}(-\ql+u^j)\left(\frac{f^i}{f^j}\right)^{\frac{s}{2}}\left((1-s)\qg_{ij}\qth_j^1-(1+s)\qg_{ji}\qth_i^1\right)\qd\qth_j^s\diff{}{\qd\qth_i^s}\\
&+\sum_{i,j}(-\ql+u^j)\frac{\qp_jf^i}{f^j}\qth_i^1\qd\qth_j^1\diff{}{\qd\qth_i^1},\\
\phi_3 =& -\frac 12\sum_{s\geq 0;i,j}(-\ql+u^j)\frac{\qp_jf^i}{f^j}\qth_j^1\left(s\qd u^{i,s}\diff{}{\qd u^{i,s}}+(s+1)\qd\qth_i^s\diff{}{\qd\qth_i^s}\right)\\
&+\sum_{i,j}(-\ql+u^j)(\qg_{ij}\qth_j^1-\qg{ji}\qth_i^1)\left(\qth_j\diff{}{\qth_i}+\qd\qth_j\diff{}{\qd\qth_i}\right)\\
&+\sum_i\qth_i^1\mathcal E_i+\sum_i(-\ql+u^i)\qth_i^1\diff{}{u^i},
\end{align*}
here $\mathcal E_i$ is an Euler-type vector filed given by
\begin{equation*}
\mathcal E_i = \sum_{s\geq 1}\left(\frac s2+1\right)u^{i,s}\diff{}{u^{i,s}}+\sum_{s\geq 0}\left(\frac s2+1\right)\qd u^{i,s}\diff{}{\qd u^{i,s}}+\sum_{s \geq 0}\frac{s-1}{2}\left(\qth_i^s\diff{}{\qth_i^s}+\qd\qth_i^s\diff{}{\qd\qth_i^s}\right).
\end{equation*}

We first simplify $\tilde \Qd_{0,1}$ by the following observation (which is parallel to Lemma 3.6 of \cite{carlet2018central}).
\begin{Lem}
Both $\phi_1$ and $\phi_2$ act trivially on $^1E_1$.
\end{Lem}
\begin{proof}
We first recall from Theorem \ref{thm-5-6-zh} that 
\begin{equation}
\label{decomp3}
^1E_1\cong \hm{C}[\ql]\{\qd\qth\}\oplus\bigoplus_{i=1}^n\frac{\hat d_i(\hm C_i[\ql])}{(-\ql+u^i)\hat d_i(\hm C_i[\ql])}\{\qd\qth\}\oplus\bigoplus_{i=1}^n\qd\hat d_i(\hm H_i).
\end{equation}
The vanishing of the action of $\phi_1$ on the cohomology 
\begin{equation*}
\hm{C}[\ql]\{\qd\qth\}\oplus\bigoplus_{i=1}^N\frac{\hat d_i(\hm C_i[\ql])}{(-\ql+u^i)\hat d_i(\hm C_i[\ql])}\{\qd\qth\}
\end{equation*}
of $\hm A[\ql]\{\qd\qth\}$ is a direct consequence of Lemma 3.6 of \cite{carlet2018central}, and the vanishing of that of $\phi_2$ is obvious. 

Next we consider action of $\phi_1$ and $\phi_2$ on
\[H(\hm A[\ql]\{\qd u^{i,s},\qd\qth_i^{s+1}\},\Qd_{-1})=\qd\hat d_i(\hm H_i)\]
for a fixed index $i$. By identifying $\ql$ with $u^i$, we can represent $\qo$ in the form
\begin{equation*}
\qo = \qd\hat d_i(\sum_{s\geq 0}p_s\qd u^{i,s}) = \sum_{s\geq 0}\hat d_i(p_s)\qd u^{i,s}+(-1)^pp_s\qd\qth_i^{s+1},\quad p_s\in\hm C_i.
\end{equation*} 
Under such an identification, the action of $\phi_1$ can be represented as
\begin{align*}
\phi_1 =& -\sum_{s\geq 1;k,j}(-u^i+u^k)\left(\frac{f^k}{f^j}\right)^{\frac{s+1}{2}}\left((s+2)\qg_{jk}\qth_k^1+s\qg_{kj}\qth_j^1\right)u^{j,s}\diff{}{u^{k,s}}\\
&-\sum_{s\geq 0;k,j}(-u^i+u^k)\left(\frac{f^k}{f^j}\right)^{\frac{s+1}{2}}\left((s+2)\qg_{jk}\qth_i^1+s\qg_{kj}\qth_j^1\right)\qd u^{j,s}\diff{}{\qd u^{k,s}}\\
&+\sum_{s\geq 2;k,j}(-u^i+u^j)\left(\frac{f^k}{f^j}\right)^{\frac{s}{2}}\left((1-s)\qg_{kj}\qth_j^1-(1+s)\qg_{jk}\qth_k^1\right)\qth_j^s\diff{}{\qth_k^s}.
\end{align*}
Hence it is clear that 
\[\phi_1\qo \in\bigoplus_{j\neq i}\hm C_j^{nt}\{\qd u^{i,s},\qd\qth_i^{s+1}\mid s\geq 0\}.\] 
For the action of $\phi_2$ on $\qo$, we first regard $\qo$ as an element of $\qd \hat d_i(\hm H_i)[\ql]$, then we observe that 
\[\phi_2\qo\in\bigoplus_{j} (-\ql+u^j)\hm C_i[\ql]\{\qd \qth_j^{s+1}\mid s\geq 0\}.\] We further make the decomposition $\phi_2\qo=\qa_1+\qa_2$, where 
\[\qa_1\in\bigoplus_{j\neq i} (-\ql+u^j)\hm C_i[\ql]\{\qd \qth_j^{s+1}\mid s\geq 0\},\quad 
\qa_2\in (-\ql+u^i)\hm C_i[\ql]\{\qd \qth_i^{s+1}\mid s\geq 0\}.\] 
Finally it is easy to see that from the definition of $\phi_3$ that 
\[\phi_3\qo\in\hm C_i[\ql]\{\qd u^{i,s},\qd\qth_i^{s+1}\mid  s\geq 0\}.\] 

 It follows from
\begin{equation*}
\Qd_{0,1}\Qd_{-1}+\Qd_{-1}\Qd_{0,1} = 0
\end{equation*}
that $\tilde \Qd_{0,1}\qo=(\phi_1+\phi_2+\phi_3)\qo$ still lies in the cohomology group. Since the subspaces
\begin{align*}
&\bigoplus_{j\neq i}\hm C_j^{nt}\{\qd u^{i,s},\qd\qth_i^{s+1}\mid s\geq 0\},\\ &\bigoplus_{j\neq i} (-\ql+u^j)\hm C_i[\ql]\{\qd \qth_j^{s+1}\mid s\geq 0\},\\ &\hm C_i[\ql]\{\qd u^{i,s},\qd\qth_i^{s+1}\mid s\geq 0\}
\end{align*}
are disjoint and they are all invariant under the action $\Qd_{-1}$, we conclude that $\phi_1\qo$, $\qa_1$ and $\qa_2+\phi_3\qo$ lie in the cohomology group.

From our computation of $^1E_1$ given in \ref{subsec-5-2-zh}, it follows that terms in the subspace 
\[\bigoplus_{j\neq i}\hm C_j^{nt}\{\qd u^{i,s},\qd\qth_i^{s+1}\mid s\geq 0\}\] 
vanish in $^1E_1$ and hence action of $\phi_1$ vanishes. Similarly, from the computation of the spectral sequence $^2E$ given in \ref{subsec-5-2-zh}, it follows that the terms in the subspace 
\[\bigoplus_{j\neq i} (-\ql+u^j)\hm C_i[\ql]\{\qd \qth_j^{s+1}\mid s\geq 0\}\] 
also vanish, hence $\qa_1 = 0$. For the same reason, the multiples of $(-\ql+u^i)$ in the subspace $\hm C_i[\ql]\{\qd u^{i,s},\qd\qth_i^{s+1}\mid s\geq 0\}$ are trivial in the cohomology as well, hence in particular $\qa_2 = 0$ which implies the vanishing of $\phi_2$. The lemma is proved.
\end{proof}

This above lemma shows that $\tilde\Qd_{0,1} = \phi_3$ and therefore each summand in the decomposition \eqref{decomp3} of $^1E_1$ 
is preserved by $\tilde \Qd_{0,1}$. In what follows, we will show that for some bidegrees $(p,d)$ the action of $\tilde\Qd_{0,1}$ is acyclic on each summand when restricted to the elements with super degree $p$ and differential degree $d$.
\begin{Lem}
\label{s5-t3}
We have $H^p_d(\hm{C}[\ql]\{\qd\qth\},\tilde\Qd_{0,1}) = 0$ unless
\begin{equation*}
d = 0,\cdots,n;\quad p = d+1,\cdots,d+n+1.
\end{equation*}
\end{Lem}
\begin{proof}
Indeed, possible bidegrees $(p,d)$ of elements of  in $\hm C[\ql]\{\qd\qth\}$ are precisely those excluded in the lemma. The lemma is proved.
\end{proof}

To compute the cohomology of $\tilde \Qd_{0,1}$ on the space
\begin{equation*}
\frac{\hat d_i(\hm C_i[\ql])}{(-\ql+u^i)\hat d_i(\hm C_i[\ql])}\{\qd\qth\},
\end{equation*}
let us first identify this space with $\hat d_i(\hm C_i)\{\qd\qth\}$ by sending $\ql$ to $u^i$. After this identification, the action of $\tilde\Qd_{0,1}$ reads (we keep the same notation)
\begin{align*}
\tilde\Qd_{0,1}=&\sum_j\qth_j^1\mathcal E_j-\frac 12\sum_{j,k}(-u^i+u^j)\frac{\qp_jf^k}{f^j}\qth_j^1\qd\qth_k\diff{}{\qd\qth_k}+\sum_j(-u^i+u^j)f^j\qth_j^1\diff{}{u^j}\\
&+\sum_{j,k}(-u^i+u^j)(\qg_{kj}\qth_j^1-\qg_{jk}\qth_k^1)\left(\qth_j\diff{}{\qth_k}+\qd\qth_j\diff{}{\qd\qth_k}\right).
\end{align*}

\begin{Lem}
\label{lem1}
We have $H^p_d(\hat d_i(\hm{C_i})\{\qd\qth\},\tilde\Qd_{0,1}) = 0$ unless
\begin{equation}
d = 2,\cdots,n+3;\quad p = d,\cdots,d+n.
\end{equation}
\end{Lem}
\begin{proof}
To compute the cohomology, we introduce a forth spectral sequence $^4E$ given by a filtration of $\hat d_i(\hm{C_i})\{\qd\qth\}$ using the $\qth_i^1$-degree, which is defined by
 $\deg_{\qth_i^1}\qth_i^1 = 1$ and by setting the degrees of other generators to be zero. By decomposing the differential $\tilde\Qd_{0,1}$
 with respect to the $\qth_i^1$-degree, we conclude that the zeroth page of this spectral sequence is given by $^4E_0 = \hat d_i(\hm{C_i})\{\qd\qth\}$ with the differential $\mathcal D_i$ given by
\begin{equation*}
\mathcal D_i = \qth_i^1\mathcal E_i+\sum_j(u^i-u^j)\qg_{ji}\qth_i^1\left(\qth_j\diff{}{\qth_i}+\qd\qth_j\diff{}{\qd\qth_i}\right).
\end{equation*}

The idea to compute $H(^4E_0, \mathcal D_i)$ is as follows. We first make the decomposition
\begin{equation*}
d_i(\hm{C_i})\{\qd\qth\} = d_i(\hm{C_i})\{\qd\qth_i\}\oplus\bigoplus_{j\neq i}d_i(\hm{C_i})\{\qd\qth_j\}.
\end{equation*}
Note that $\bigoplus_{j\neq i}d_i(\hm{C_i})\{\qd\qth_j\}$ is an invariant subspace of $\mathcal D_i$ while $d_i(\hm{C_i})\{\qd\qth_i\}$ is not. Nevertheless, if we can find proper bidegrees $(p,d)$ such that $\mathcal D_i$ is acyclic on $\bigoplus_{j\neq i}d_i(\hm{C_i})\{\qd\qth_j\}\cap\Qo^p_d$, then from Lemma \ref{ha-lem} we know that the cohomology of $d_i(\hm{C_i})\{\qd\qth\}\cap\Qo^p_d$ is given by the cohomology of the space $d_i(\hm{C_i})\{\qd\qth_i\}\cap\Qo^p_d$ with the differential given by the projection of $\mathcal D_i$.

Take any monomial $\mathfrak m$ in $u^{i,s},\qth_i^{s+1}$ for $s\geq 1$ and any monomial $g\in\hm C$. For $j\neq i$, we have
\begin{equation*}
\mathcal D_i(g\hat d_i(\mathfrak m)\qd\qth_j) = \qth_i^1\biggl(\sum_j(u^i-u^j)\qg_{ji}\qth_j\diff{g}{\qth_i}+(w_i(g)+w_i(\mathfrak m)-1)g\biggr)\hat d_i(\mathfrak m)\qd\qth_j,
\end{equation*}
here for a monomial $\qo\in\Omega$, the rational number $w_i(\qo)$ is defined by ${\mathcal E}_i \qo=w_i(\qo)\qo$.

So for any fixed $\mathfrak m$ and $j\neq i$, the subspace $\hm C\hat d_i(\mathfrak m)\{\qd\qth_j\}$ is invariant under the action of $\mathcal D_i$ , and we can concentrate first on computing the cohomology of this subspace. 
The following argument is basically the same as that of Lemma 3.10 of \cite{carlet2018central}, nonetheless we still write it down for the convenience of readers. Let us decompose $\hm C$ in the form 
\begin{equation*}
\hm C = \hm C^i_0\oplus\qth_i\hm C_0^i,
\end{equation*}
where $\hm C^i_0$ is the subspace of $\hm C$ spanned by monomials that do not contain $\qth_i$. For $g \in \hm C^i_0$, the action of $\mathcal D_i$ is given by
\begin{equation}\label{diform-zh}
\mathcal D_i(g\hat d_i(\mathfrak m)\qd\qth_j) = \qth_i^1(w_i(\mathfrak m)-1)g\hat d_i(\mathfrak m)\qd\qth_j,
\end{equation}
therefore the subspace $\hm C^i_0\hat d_i(\mathfrak m)\{\qd\qth_j\}$ is acyclic. Indeed, since $w_i(\mathfrak m)\ge\frac32$ if $\hat d_i(\mathfrak m) \neq 0$, so a nonzero cocycle of $\mathcal D_i$ must contain $\qth_i^1$. Assume we have a cocycly $\qth_i^1h$ for some $h$, then it follows from \eqref{diform-zh} the existence of a suitable constant $c$ such that $\qth_i^1h = c\mathcal D_i(h)$.

Due to Lemma \ref{ha-lem}, in order to compute the cohomology of $\hm C\hat d_i(\mathfrak m)\{\qd\qth_j\}$ we only need to compute the one for the subspace $\qth_i\hm C_0^i\{\qd\qth_j\}$. The projection of $\mathcal D_i$ on $\qth_i\hm C_0^i\hat d_i(\mathfrak m)$ is just a multiplication by $\qth_i^1(w_i(\mathfrak m)-\frac{3}{2})$, which is acyclic if $w_i(\mathfrak m)\neq\frac{3}{2}$. So the nontrivial cocycles are given by elements of $\qth_i\hm C_0^i\hat d_i( u^{i,1})\qd\qth_j = \hm C_0^i\qth_i\qth_i^2\qd\qth_j$, which have the following possible bidegrees: 
\begin{equation*}
d = 2,\cdots,2+n;\quad p = d+1,\cdots,d+n.
\end{equation*}
Thus unless a bidegree $(p,d)$ takes the values given above, the subcomplex 
\[\bigoplus_{j\neq i}d_i(\hm{C_i})\{\qd\qth_j\}\cap\Qo^p_d\] 
is acyclic. 

Thus to compute the cohomology of ${^4 E_0}$, we only need to consider the cohomology of the space $d_i(\hm{C_i})\{\qd\qth_i\}$ due to Lemma \ref{ha-lem}. The differential is the projection of $\mathcal D_i$
which can be represented as
\begin{equation*}
\qth_i^1\mathcal E_i+\sum_j(u^i-u^j)\qg_{ji}\qth_i^1\qth_j\diff{}{\qth_i}.
\end{equation*}
Now by repeating the argument above we see that the nontrivial cocycles are of the form $\hm C^i_0\qth_i^2\qd\qth_i$ or $\hm C^i_0\qth_i\qth_i^3\qd\qth_i$. By counting the possible bidegrees of these elements, we complete the proof of the lemma.
\end{proof}
\begin{Rem}
By a more careful analysis, we can prove that actually 
\[H^{2n+3}_{n+3}(\hat d_i(\hm{C_i})\{\qd\qth\},\tilde\Qd_{0,1}) = 0.\]
But this is not important for our consideration of the deformation problem.
\end{Rem}

Finally we are to compute the cohomology of $\tilde \Qd_{0,1}$ on the space $\qd\hat d_i(\mathcal H_i)$. Recall that on this space we have to identify $\ql$ with $u^i$, and therefore the differential reads
\begin{align*}
\tilde \Qd_{0,1} =& -\frac 12\sum_{s\geq 0;k,j}(-u^i+u^j)\frac{\qp_jf^k}{f^j}\qth_j^1\left(s\qd u^{k,s}\diff{}{\qd u^{k,s}}+(s+1)\qd\qth_k^s\diff{}{\qd\qth_k^s}\right)\\
&+\sum_{k,j}(-u^i+u^j)(\qg_{kj}\qth_j^1-\qg_{jk}\qth_k^1)\qth_j\diff{}{\qth_k}+\sum_j\qth_j^1\mathcal E_j+\sum_j(-u^i+u^j)\qth_j^1\diff{}{u^j}.
\end{align*}
\begin{Lem}
\label{s5-t4}
We have $H^p_d(\qd\hat d_i(\mathcal H_i),\tilde \Qd_{0,1}) = 0$ unless
\begin{equation*}
d = 3,\cdots,n+3;\quad p = d,\cdots d+n-1.
\end{equation*}
\end{Lem}
\begin{proof}
The idea is very similar to the proof of Lemma \ref{lem1}. We introduce another spectral sequence by filtrating $\qd\hat d_i(\mathcal H_i)$ using $\deg_{\qth_i^1}$. The differential on the zeroth page reads
\begin{equation*}
\mathcal D_i = \qth_i^1\mathcal E_i+\sum_j(u^i-u^j)\qg_{ji}\qth_i^1\qth_j\diff{}{\qth_i}.
\end{equation*}
Let us denote $\psi = \sum_j(u^i-u^j)\qg_{ji}\qth_i^1\qth_j\diff{}{\qth_i}$. For any monomial $\mathfrak m$ in $u^{i,s},\qth_i^{s+1}$ with $s\geq 1$, and any monomial $g\in\hm C$, we have
\begin{align*}
\mathcal D_i\qd\hat d_i(g\mathfrak m\qd u^{i,s}) =& \mathcal D_i(g\hat d_i(\mathfrak m)\qd u^{i,s}+(-1)^pg\mathfrak m\qd\qth_i^{s+1})\\
=&\qth_i^1\left(w_i(g)+w_i(\mathfrak m)-1+\frac s2 +1\right)g\hat d_i(\mathfrak m)\qd u^{i,s}\\
&+(-1)^p\qth_i^1\left(w_i(g)+w_i(\mathfrak m)+\frac s2\right)g\mathfrak m\qd\qth_i^{s+1}\\
&+\qth_i^1\left(\psi(g)\hat d_i(\mathfrak m)\qd u^{i,s}+(-1)^p\psi(g)\mathfrak m\qd\qth_i^{s+1}\right)\\
=&-\qd\hat d_i\left(\qth_i^1\psi(g)\mathfrak m\qd u^{i,s}+\qth_i^1\left(w_i(g)+w_i(\mathfrak m)+\frac s2\right)g\mathfrak m\qd u^{i,s}\right).
\end{align*}
So the subspace $\hm C\qd\hat d_i(\mathfrak m\qd u^{i,s})$ is invariant under the action of $\mathcal D_i$. Similarly, we make the decomposition  $\hm C = \hm C^i_0\oplus\qth_i\hm C^i_0$. For $g\in\hm C_i^0$, it is easy to see that
\begin{equation*}
\mathcal D_i\qd\hat d_i(g\mathfrak m\qd u^{i,s})  = -\qd\hat d_i\left(\qth_i^1\left(w_i(\mathfrak m)+\frac s2\right)g\mathfrak m\qd u^{i,s}\right).
\end{equation*}
On the other hand, for a monomial in $u^{i,s},\qth_i^{s+1}$ with $s\geq 1$, we must have that $w_i(\mathfrak m)\geq \frac 12$, thus the subspace $\hm C^i_0\qd\hat d_i(\mathfrak m\qd u^{i,s})$ is acyclic. By applying Lemma \ref{ha-lem} again, we know that we only need to consider the cohomology of $\qth_i\hm C^i_0\qd\hat d_i(\mathfrak m\qd u^{i,s})$ with the differential being the projection of $\mathcal D_i$ given by the multiplication of
\begin{equation*}
-\qth_i^1\left(w_i(\mathfrak m)+\frac s2-\frac 12\right).
\end{equation*}
So the only possible nontrivial cocycle is given by the case when $s = 0$ and $w_i(\mathfrak m) = \frac 12$, i.e. when the monomial $\mathfrak m$ is of the form
\begin{equation*}
\qth_i\hm C^i_0\qd\hat d_i(\qth_i^2\qd u^{i}) = \hm C^i_0\qth_i\qth_i^2\qd\qth_i^1.
\end{equation*}
By counting the possible bidegrees of such elements, we complete the proof of the lemma.
\end{proof}

Let us summarize all the results obtained above. We first construct a spectral sequence $^1E$ to compute the cohomology group $H^p_d(\Qo[\ql],\qp_\ql)$. Then we transform the computation to finding suitable bidegrees $(p,d)$ such that the second page $^1E_2 = 0$ when restricted to $\Qo^p_d$. To this end we introduce a third spectral sequence $^3E$, and we conclude that the first page $^3E_1$ vanishes for suitable bidegrees in Lemmas \ref{s5-t3}--Lemma \ref{s5-t4}, thus $^3E$ converges to $^1E_2$ and so $^1E_2$ vanishes. Consequently $^1E$ converges to $H^p_d(\Qo[\ql],\qp_\ql)$. In this way, we prove  Theorem \ref{thm-van}.

The following proposition is an illustration of applications of the variational bihamiltonian cohomology.
\begin{Prop}
\label{biham-flow}
Let $(P_0, P_1)$ be a semisimple bihamiltonian structure of hydrodynamic type, and let $X\in\mathrm{Der}(\hm A)_{1}^{0}$ commute with $\qp_x$, $D_{P_0}$ and $D_{P_1}$. Then for any deformation $(\tilde P_0,\tilde P_1)$ of $(P_0, P_1)$, there exists a unique $\tilde X \in\mathrm{Der}(\hm A)^0_{\geq 1}$ with leading term given by $X$ such that $\tilde X$ commutes with $\qp_x$, $D_{\tilde P_0}$ and $D_{\tilde P_1}$.
\end{Prop}
\begin{proof}
Let us decompose $\tilde X$ and $\tilde P_a$ according to the differential degree as follows:
\begin{align*}
\tilde X &= X^{[0]}+\sum_{k\geq 1}X^{[k]},\quad X^{[k]}\in \mathrm{Der}(\hm A)_{k+1}^{0},\ X^{[0]} = X;\\
\tilde P_a &= P_a+\sum_{k\geq 1}P_a^{[k]},\quad P_a^{[k]}\in\hm F^2_{k+1},\quad a=0,1.
\end{align*}
The condition that $\tilde X$ commutes with $D_{\tilde P_a}$ is equivalent to the following equations:
\begin{equation}\label{eq-hvf-zh}
\fk{D_{ P_a}}{X^{[k]}}+\sum_{l=1}^{k}\fk{D_{\tilde P_a^{[l]}}}{X^{[k-l]}} = 0,\quad k\geq 1,\ a = 0,1.
\end{equation}

To prove the uniqueness of $\tilde X$, we only need to show that if $X = 0$ then $\tilde X = 0$. Indeed, if $X = 0$, then we have the following equation for $X^{[1]}$:
\[\fk{D_{P_a}}{X^{[1]}} = 0,
\]
which is equivalent to $\tilde D_a\mathcal X^{[1]} = 0$. Here 
$\mathcal X^{[k]} \in\bar\Qo^1_{k+1}$ is the 1-form corresponds to $X^{[k]}$.
Recall that 
\[
\vbh^1_{d}(\bar\Qo,\tilde D_0,\tilde D_1) = \bar\Qo^1_d\cap\ker\tilde D_0\cap\ker\tilde D_1,
\]
then by the vanishing of $\vbh^1_{2}(\bar\Qo,\tilde D_0,\tilde D_1)$ given by Theorem \ref{vbh12}, we conclude that $X^{[1]} = 0$. In a similar way,  we can prove recursively that $X^{[k]} = 0$ and the uniqueness is proved.

The proof of the existence is very similar to the one given in Sect.\,\ref{e-deform} and we omit the details here. For a sketch of the proof, we may first assume $P_a^{[1]} = 0$ and hence we can choose $X^{[1]} = 0$, then by using $\vbh^2_{\geq 4}(\bar\Qo,\tilde D_0,\tilde D_1) = 0$, we can recursively solve the equations \eqref{eq-hvf-zh} to obtain $X^{[\geq 2]}$. The proposition is proved.
\end{proof}
\begin{Rem}
This result is a generalization of the fact that a bihamiltonian vector filed is uniquely determined by its leading term. This generalized version can be applied to the case when the flows are not in the space $\mathrm{Der}(\hm A)^D$. Typical examples of such kind of flows are given by Virasoro symmetries.
\end{Rem}

\section{Conformal bihamiltonian structures}
\label{conformal}

\subsection{Conformal Bihamiltonian Structures of hydrodynamic type}\label{cbs}
This section is devoted to the proof of the Theorem \ref{g0-conf}. To this end, we consider a semisimple bihamiltonian structure $(P_0,P_1)$ of hydrodynamic type which is represented as in \eqref{norm-p}. 

Assume that the bihamiltonian structure $(P_0,P_1)$ is conformal, we are going to find a derivation $E$ such that it satisfied the equation \eqref{eq-def-2-1-zh}. We  first make the following decomposition 
\[E = E^{[0]}+E^{[\geq 1]},\quad E^{[0]}\in\mathrm{Der}(\hm A)^0_0,\,E^{[\geq 1]}\in\mathrm{Der}(\hm A)^0_{\geq 1}.\]
According to the equation \eqref{eq-def-2-1-zh}, it is easy to see that $E^{[\geq 1]}$ is a vector field that commutes with $D_{P_0}$ and $D_{P_1}$ and thus it is an element of $\vbh^1_{\ge 1}(\bar\Qo,\tilde D_0,\tilde D_1)$. In what follows we assume that $E\in\mathrm{Der}(\hm A)^0_0$.

Denote $E(u^i) = F^i$ and $E(\qth_i) = \sum_j G^j_i\qth_j$, where $F^i$ and $G^j_i$ are some smooth functions  of $u^1,\dots, u^n$. We are to solve these functions from the equations given in \eqref{eq-def-2-1-zh}. 

Firstly, from the equation $[E, \qp_x] = \mu\qp_x$ it follows that 
\[
E(u^{i,s}) = \qp_x^s(F^i)+s\mu u^{i,s},\quad E(\qth^{s}_i) = \sum_j \qp_x^s(G^j_i\qth_j)+s\mu \qth_i^s.
\]
By comparing the $\qth_j^1$ coefficients $(j\neq i)$ of both sides of the equation 
\begin{equation}\label{eq-6-1-zh}
[E, D_{P_a}]u^i = \ql_aD_{P_a}(u^i)
\end{equation}
for $a = 0,1$
we obtain the following equations:
\begin{align*}
&f^iG^j_i-(\qp_jF^i)f^j = 0,\quad j \neq i\\
&u^if^iG^j_i-(\qp_jF^i)u^jf^j = 0,\quad j \neq i.
\end{align*}
From these equations it follows that $G^j_i = 0$ for $j\neq i$. We also compare the $\qth_i^1$ coefficients of both sides of the equation \eqref{eq-6-1-zh} to obtain 
\begin{align*}
&\sum_j\aij ji F^j+f^i(G^i_i+\mu)-\qp_iF^if^i = \ql_0f^i,\\
&\sum_j\qp_j(u^if^i) F^j+u^if^i(G^i_i+\mu)-\qp_iF^iu^if^i = \ql_1u^if^i.
\end{align*}
We solve these equations to arrive at
\[
F^i = (\ql_1-\ql_0)u^i,\quad G^i_i = \ql_1-\mu-\frac{1}{f^i}\sum_j(\ql_1-\ql_0)u^j\qp_jf^i.
\]
Together with the fact that $G^j_i = 0$ for $j\neq i$, we conclude that 
\[
E(u^i) = (\ql_1-\ql_0)u^i,\quad E(\qth_i) = \left(\ql_1-\mu-\frac{1}{f^i}E(f^i)\right)\qth_i.
\]
For simplicity, we will denote $E(\qth_i) = G^i\qth_i$.

Next we compare the $\qth_i$ coefficients of the equation 
\begin{equation}\label{eq-6-2-zh}
[E, D_{P_0}]u^i = \ql_0D_{P_0}(u^i)
\end{equation}
to arrive at the following equation: 
\begin{align*}
\sum_{j,k}&\frac{f^i\qp_j\qp_kf^i-\aij ji\aij ki}{f^i}(\ql_0-\ql_1)u^ku^{j,1}+\sum_j \aij ji(\ql_0-\ql_1)u^{j,1}\\
&+\frac 12 \sum_{j,k}\qp_k\qp_jf^i(\ql_1-\ql_0)u^ku^{j,1}+\frac 12\sum_j \aij jiu^{j,1}G^i = \frac 12\sum_j(\ql_0-\mu)\aij ji u^{j,1}.
\end{align*}
By further comparing the $u^{j,1}$ coefficients for both sides of the above equation, one can show that the above equation is actually equivalent to 
\begin{equation}
\label{s2-t1}
E\left(\frac{\aij ji}{f^i}\right) = (\ql_0-\ql_1)\left(\frac{\aij ji}{f^i}\right),\quad \forall i,j.
\end{equation}
For $j\neq i$, compare the coefficients of $u^{i,1}\qth_j$ and $u^{j,1}\qth_j$ on both sides of the equation \eqref{eq-6-2-zh}, one obtains the following equations:
\begin{align}
\label{s2-t2}
&E\left(\bij ij\right)+\bij ij G^j = (\ql_0-\mu)\bij ij, \quad j\neq i.\\
\label{s2-t3}
&E\left(\bij ji\right)+\bij ji G^j = (\ql_0-\mu)\bij ji, \quad j\neq i.
\end{align}

 Finally let us compare the coefficients on both sides of the equation 
 \[[E, D_{P_0}]\qth_i = \ql_0D_{P_0}(\qth_i).\] 
 For $j\neq i$, we consider first the coefficients of $\qth_i\qth_j^1$ and obtain that
\begin{equation}
\label{s2-t4}
E\left(\bij ji\right)+\bij ji G^j = (\ql_0-\mu)\bij ji-\qp_jG^if^j, \quad j\neq i.
\end{equation}
Therefore it is easy to conclude that $\qp_jG^i = 0$ for $j\neq i$ by comparing \eqref{s2-t3} and \eqref{s2-t4}. We also compare the coefficients of $\qth_i\qth_i^1$, after a straightforward computation, we arrive at an equation
\begin{equation}
\label{s2-t5}
E\left(\frac{\aij ii}{f^i}\right) = (\ql_0-\ql_1)\left(\frac{\aij ii}{f^i}\right)-\qp_iG^i.
\end{equation}
Hence we conclude that $\qp_iG^i = 0$ by comparing \eqref{s2-t5} with \eqref{s2-t1}, and $G^i$ must be a constant. This means that there exists real numbers $\qa^i$ such that $E(f^i) = \qa_i f^i$, which gives the condition \eqref{hom-f} with $d^i$ determined by $\qa^i = (\ql_1-\ql_0)d^i$. Substitute the expression for $G^j$ into \eqref{s2-t2}, we obtain the condition \eqref{irre-f}, and hence the `only if' part of the Theorem \ref{g0-conf} is proved.

The `if' part of the Theorem \ref{g0-conf} can be checked easily by a straightforward computation and hence the theorem is proved.

\subsection{Deformed Conformal Bihamiltonian Structures}\label{e-deform}
In this section, we will use the theory of variational bihamiltonian cohomology to prove Theorem \ref{g1-conf}. Let $(P_0,P_1)$ be a conformal semisimple bihamiltonian structure of hydrodynamic type as described in Theorem \ref{g0-conf} and $(\tilde P_0, \tilde P_1)$ be any of its deformation. We are going to find a deformation $\tilde E\in\mathrm{Der}(\hm A)^0$, such that 
\begin{equation}
\label{def-tildeE}
\left[\tilde E, \qp_x\right] = \mu\qp_x;\quad \left[\tilde E, D_{\tilde P_a}\right] = \ql_aD_{\tilde P_a},\quad a = 0,1.
\end{equation}

We first decompose $\tilde P_a$ and $\tilde E$ as follows:
\[
\tilde P_a = \sum_{k\geq 0}\tilde P_a^{[k]},\ \tilde P_a^{[k]}\in\hm F^2_{k+1};\quad \tilde E = \sum_{k\geq 0}E^{[k]},\ E^{[k]}\in\mathrm{Der}(\hm A)^0_k.
\]
We also make the same decomposition for both sides of the equation $[\tilde E, D_{\tilde P_a}] = \ql_aD_{\tilde P_a}$ to obtain
\begin{equation}
\label{s6-t1}
\sum_{i=0}^l\fk{E^{[i]}}{D_{P_a^{[l-i]}}} = \ql_aD_{P_a^{[l]}},\quad a = 0,1;\ l\geq 1.
\end{equation}
According to the result of bihamiltonian cohomology, we can assume (by doing a Miura transformation if necessary) that $P_a^{[1]} = 0$, hence we can choose $E^{[1]} = 0$. Let us proceed to find the deformations $E^{[\geq 2]}$. Let us denote
\[
W_a^{[l]} = \sum_{i=0}^{l-1}\fk{E^{[i]}}{D_{P_a^{[l-i]}}}-\ql_aD_{P_a^{[l]}},
\]
then we can recursively determine the deformation by solving the equation
\[
\fk{D_{P_a^{[0]}}}{E^{[l]}} = W_a^{[l]}.
\]
\begin{Lem}
\label{s6-close}
Assume that we have found $E^{[i]}$ for all $i<k$ such that the equations in \eqref{s6-t1} are satisfied for $l<k$. Then we have the following identity:
\[
\fk{W_b^{[k]}}{D_{P_a^{[0]}}}+\fk{W_a^{[k]}}{D_{P_b^{[0]}}} = 0,\quad a, b = 0, 1.
\]
\end{Lem}
\begin{proof}
The proof is a straightforward computation by using the equations
\[
\fk{E^{[i]}}{D_{P_a^{[0]}}} = W_a^{[i]},\, i<k;\quad \fk{D_{\tilde P_a}}{D_{\tilde P_b}} = 0,
\]
and the (graded) Jacobi identities.
\end{proof}

In order to apply the theory of variational bihamiltonian cohomology, we need the following fact.
\begin{Lem}We have
$\fk{E^{[l]}}{\qp_x} = 0$ and $\fk{W_a^{[l]}}{\qp_x} = 0$ for $l\geq 1$.
\end{Lem}
\begin{proof}
This follows directly from the assumption $[\tilde E,\qp_x] = \mu\qp_x$.
\end{proof}
As a consequence of the above lemma, we can regard the vector fields $E^{[l]}$ and $W_a^{[l]}$ as elements of the space $\bar\Qo$. We denote the corresponding 1-forms by $\mathcal E^{[l]}$ and $\mathcal W_a^{[l]}$ respectively. More precisely, we have
\[
\mathcal E^{[l]}\in\bar\Qo^1_{l},\quad \mathcal W_a^{[l]}\in\bar\Qo^2_{l+1}.
\]

Now the equations we are to solve can be rewritten in the form
\begin{equation}\label{eq-6-3-zh}
\tilde D_a\mathcal E^{[l]} = \mathcal W_a^{[l]},\quad a = 0,1,\ l\geq 2,
\end{equation}
and Lemma \ref{s6-close} gives the following conditions
\[
\tilde D_a\mathcal W_a^{[l]} = 0,\quad a = 0,1.
\]
Thus by using the triviality of the variational Hamiltonian cohomology we can find an element $\qg^{[l]}\in\bar\Qo^1_l$ such that
\[
\tilde D_0 \qg^{[l]}= \mathcal W_0^{[l]}.
\]
Then the solution of $\mathcal E^{[l]}$ can be represented by
\[
\mathcal E^{[l]} = \qg^{[l]}+\tilde D_0 \qa^{[l]},\quad \qa^{[l]}\in\Qo^0_{l-1}.
\]
The 1-form $\qa^{[l]}$ should be determined by the equation
\begin{equation*}
\tilde D_1\mathcal E^{[l]} = \tilde D_1\kk{\qg^{[l]}+\tilde D_0 \qa^{[l]}} = \mathcal W_1^{[l]}.
\end{equation*}
 By using Lemma \ref{s6-close} again, we see that $\mathcal W_1^{[l]}-\tilde D_1\qg^{[l]}$ lies in $\ker\tilde D_0\cap\ker\tilde D_1$. Therefore from the fact that $\vbh^2_{\geq 4}(\bar\Qo,\tilde D_0,\tilde D_1) = 0$ it follows that we can always solve the above equation to obtain $\qa^{[l]}$ for $l\geq 3$. 

Now let us try to find $\mathcal E^{[2]}$ by solving the equations in \eqref{eq-6-3-zh} for $l=2$. We will work in the canonical coordinates $u^1,\dots, u^n$ of $(P_0,P_1)$. According to the results in \cite{liu2005deformations}, we can choose a Miura type transformation  such that $\tilde P_0 = P_0$, $P_1^{[1]}=0$  and the derivation $D_{P_1^{[2]}}$ is given by the 1-form $\tilde D_0\mathcal T$, where $\mathcal T$ is given by
\[
\mathcal T = \int\qd\kk{D_{P_1}\sum_ic_i(u^i)u^{i,1}\log u^{i,1}-D_{P_0}\sum_iu^ic_i(u^i)u^{i,1}\log u^{i,1}}.
\]
Here the functions $c_i(u^i)$ are the central invariants of the deformed bihamiltonian structure $(\tilde P_0, \tilde P_1)$ . Since $P_0^{[2]} = 0$, we can choose $\qg^{[2]} = 0$. Then we have 
\[
W_1^{[2]} = \fk{E^{[0]}}{D_{P_1^{[2]}}}-\ql_1D_{P_1^{[2]}},
\]
here $E^{[0]}=E$ is described as in Theorem \ref{g0-conf}. Let us use $T$ to denote the derivation given by the 1-form $\mathcal T$, then we have
\begin{align*}
W_1^{[2]}&=\fk{E^{[0]}}{\fk{D_{P_0}}{T}}-\ql_1\fk{D_{P_0}}{T}\\
&=-\fk{D_{P_0}}{\fk{T}{E^{[0]}}}-\fk{T}{\fk{E^{[0]}}{D_{P_0}}}-\ql_1\fk{D_{P_0}}{T}\\
&=-\fk{D_{P_0}}{\fk{T}{E^{[0]}}}+\ql_0\fk{D_{P_0}}{T}-\ql_1\fk{D_{P_0}}{T}\\
&=\fk{D_{P_0}}{(\ql_0-\ql_1)T-\fk{T}{E^{[0]}}}.
\end{align*}
Denote by $\qb\in\bar\Qo^1_2$ the 1-form corresponding to the derivation $(\ql_0-\ql_1)T-\fk{T}{E^{[0]}}$, then we see that
\[
\tilde D_1\mathcal E^{[2]} = \tilde D_1\tilde D_0\qa^{[2]} = \mathcal W_1^{[2]} = \tilde D_0 \qb.
\]
To solve this equation for $\mathcal E^{[2]}$, we need to check that $[\tilde D_0\qb]\in \vbh^2_3(\bar\Qo,\tilde D_0,\tilde D_1)$ is trivial. According to Lemma \ref{s3-t8} we only need to check that the indices $ind_i(\qb)$ for $i = 1,\cdots,n$ vanish.

We first note that 
\[ind_i(\mathcal T) = -3c_i(u^i).\]
In another word, if we represent $\mathcal T$ in the form
\[
\mathcal T = \int \sum_i X^i\qd u^i+Y^i\qd\qth_i,\quad X^i\in\hm A^1_2,\quad Y^i\in\hm A^0_2,
\]
where $X^i$ and $Y^i$ are given by
\begin{align*}
X^i &= \sum_j X^{(i)}_{j}\qth_j^2+\sum_{j,k}\left(X^{(i)}_{kj}u^{j,1}\qth_k^1+Z^{(i)}_{jk}u^{k,2}\qth_j\right)+\sum_{j,k,l}Z^{(i)}_{j;kl}u^{k,1}u^{l,1}\qth_j,\\
Y^i &= \sum_j Y^{(i)}_ju^{j,2}+\sum_{j,k}Y^{(i)}_{jk}u^{j,1}u^{k,1},
\end{align*}
then we must have $X^{(i)}_i+Y^{(i)}_i = -3c^i(u^i)f^i$.
From the explicit formula \eqref{gen-e} for $E^{[0]}$ it follows that
\begin{align*}
\fk{T}{E^{[0]}}u^i =& (\ql_1-\ql_0)X^{(i)}_iu^{i,2}-E^{[0]}\kk{X^{(i)}_i}u^{i,2}\\&-X^{(i)}_i(\ql_1-\ql_0+2\mu)u^{i,2}+\cdots\\=&-2\mu X^{(i)}_iu^{i,2}-E^{[0]}\kk{X^{(i)}_i}u^{i,2}+\cdots;\\
\fk{T}{E^{[0]}}\qth_i =& -\kk{\ql_1-(\ql_1-\ql_0)d^i-\mu}Y^{(i)}_i\qth_i^2+E^{[0]}\kk{Y^{(i)}_i}\qth_i^2\\&+\kk{\ql_1-(\ql_1-\ql_0)d^i+\mu}Y^{(i)}_i\qth_i^2+\cdots\\=&2\mu Y^{(i)}_i\qth_i^2+E^{[0]}\kk{Y^{(i)}_i}\qth_i^2+\cdots.
\end{align*}
Here we omit all the terms that do not contribute to the computation of indices. Then by the definition of the index we conclude that 
\[
ind_i(\qb) =3\kk{\ql_1-\ql_0-2\mu-(\ql_1-\ql_0)d^i}c_i(u^i)-3E^{[0]}\kk{c_i(u^i)}. 
\]
The equation $ind_i(\qb) = 0$ is actually an ODE for $c_i(u^i)$, which can 
be easily solved to give the solution
\[c_i(u^i) = C_i(u^i)^{m_i},\quad m_i = \frac{\ql_1-\ql_0-2\mu-(\ql_1-\ql_0)d^i}{\ql_1-\ql_0},\]
where $C_i$ are arbitrary constants. 
Thus we prove the Theorem \ref{g1-conf}.
\section{Conclusion}\label{con}
In this paper, we propose a generalization of the bihamiltonian cohomology, called the variational bihamiltonian cohomology, to deal with more general bihamiltonian flows. The eventual goal for developing the theory of variational bihamiltonian cohomology is to prove the following conjectures.
\begin{Conj}
For any deformation of the bihamiltonian structure of the Principal Hierarchy associated to a semisimple Frobenius manifold with constant central invariants, the Virasoro symmetries $\diff{}{s_m}$ can be deformed to be  symmetries of the deformed integrable hierarchy.
\end{Conj}

The above conjecture is verified for the case $m = -1$ in \cite{dubrovin2018bihamiltonian}. We will study this conjecture in detail and give a proof of it in the second paper \cite{variational2} of this series of papers. The proof is based on the variational bihamiltonian cohomology theory established in this paper and the construction of the super tau-covers for any tau-symmetric bihamiltonian integrable hierarchies which are generalizations of the results given in \cite{liu2020super}.

Among all the possible deformations of the bihamiltonian structure with constant central invariants, there is a particular deformation with all the central invariants being $\frac{1}{24}$. Such a deformation is called the topological deformation, and it is conjectured that:
\begin{Conj}
When the central invariants of the deformation of the bihamiltonian structure of the Principal Hierarchy are all equal to $\frac{1}{24}$, the corresponding deformation of Virasoro symmetries can be represented by linear actions on the tau function.
\end{Conj}

Apart from the above-mentioned goal, there are some other interesting problems concerning the variational bihamiltonian cohomology itself. For example, the natural map
\[
\qd:\hm F\to\bar\Qo
\]
induces an isomorphism from the bihamiltonian cohomology to the variational bihamiltonian cohomology for most of the bidegrees $(p,d)$, in particular, for the bidegress where both the bihamiltonian cohomology and the variational bihamiltonian cohomology vanish and for the bidegree $(p,d) = (2, 3)$. Then it is natural to ask if this map is indeed a quasi-isomorphism of the complex, and if so can we obtain a homotopy inverse? Also, we can generalize the notion of variational bihamiltonian cohomology to include all the possible differential forms on $J^\infty(\hat M)$, then we have:
\xym{\hm F\ar[r]^\qd &\bar{\mathcal E}^1
\ar[r]^\qd &\bar{\mathcal E}^2\ar[r]^\qd &\bar{\mathcal E}^3\ar[r]^\qd&\cdots}
It is then interesting to consider the corresponding bihamiltonian cohomology on each $\bar{\mathcal E}^k$ and to ask if each $\qd$ is an quasi-isomorphism. We conjecture that the cohomology on $\bar{\mathcal E}^k$ is related to the general $\mathcal W$-symmetry of order $k+1$. In another word, the cohomology on $\hm F$ controls the flows that do not explicitly contain time variables, and the cohomology on $\bar{\mathcal E}^k$ controls the flows that explicitly depend on at most $k$ time variables.

\medskip

\noindent Si-Qi Liu,

\noindent Department of Mathematical Sciences, Tsinghua University \\ 
Beijing 100084, P.R.~China\\
liusq@tsinghua.edu.cn
\medskip

\noindent Zhe Wang,

\noindent Department of Mathematical Sciences, Tsinghua University \\ 
Beijing 100084, P.R.~China\\
zhe-wang17@mails.tsinghua.edu.cn
\medskip

\noindent Youjin Zhang,

\noindent Department of Mathematical Sciences, Tsinghua University \\ 
Beijing 100084, P.R.~China\\
youjin@tsinghua.edu.cn


\begin{thebibliography}{10}

\bibitem{buryak2015double}
{\sc Buryak, A.}
\newblock Double ramification cycles and integrable hierarchies.
\newblock {\em Comm. Math. Phys. 336}, 3 (2015), 1085--1107.

\bibitem{buryak2012deformations}
{\sc Buryak, A., Posthuma, H., and Shadrin, S.}
\newblock On deformations of quasi{-Miura transformations and the
  Dubrovin-Zhang bracket}.
\newblock {\em J. Geom. Phys. 62}, 7 (2012), 1639--1651.

\bibitem{buryak2012polynomial}
{\sc Buryak, A., Posthuma, H., and Shadrin, S.}
\newblock A polynomial bracket for the {Dubrovin-Zhang} hierarchies.
\newblock {\em J. Differential Geom. 92}, 1 (2012), 153--185.

\bibitem{carlet2018central}
{\sc Carlet, G., Kramer, R., and Shadrin, S.}
\newblock Central invariants revisited.
\newblock {\em J. {\'E}c. polytech. Math. 5\/} (2018), 149--175.

\bibitem{carlet2018deformations}
{\sc Carlet, G., Posthuma, H., and Shadrin, S.}
\newblock Deformations of semisimple{ Poisson} pencils of hydrodynamic type are
  unobstructed.
\newblock {\em J. Differential Geom. 108}, 1 (2018), 63--89.

\bibitem{dijkgraaf1991notes}
{\sc Dijkgraaf, R., Verlinde, H., and Verlinde, E.}
\newblock Notes on topological string theory and 2d quantum gravity.
\newblock In {\em String theory and quantum gravity}. 1991.

\bibitem{dijkgraaf1991topological}
{\sc Dijkgraaf, R., Verlinde, H., and Verlinde, E.}
\newblock Topological strings in $d<1$.
\newblock {\em Nuclear Phys. B 352}, 1 (1991), 59--86.

\bibitem{dubrovin1993integrable}
{\sc Dubrovin, B.}
\newblock Integrable systems and classification of 2-dimensional topological
  field theories.
\newblock In {\em Integrable systems}. Springer, 1993, pp.~313--359.

\bibitem{dubrovin1996geometry}
{\sc Dubrovin, B.}
\newblock Geometry of {2D} topological field theories.
\newblock In {\em Integrable systems and quantum groups}. Springer, 1996,
  pp.~120--348.

\bibitem{DLZ-1}
{\sc Dubrovin, B., Liu, S.-Q., and Zhang, Y.}
\newblock On{ Hamiltonian perturbations of hyperbolic systems of conservation
  laws I: Quasi-Triviality of bi-Hamiltonian} perturbations.
\newblock {\em Comm. Pure Appl. Math. 59}, 4 (2006), 559--615.

\bibitem{dubrovin2018bihamiltonian}
{\sc Dubrovin, B., Liu, S.-Q., and Zhang, Y.}
\newblock Bihamiltonian cohomologies and integrable hierarchies {II}: the tau
  structures.
\newblock {\em Comm. Math. Phys. 361}, 2 (2018), 467--524.

\bibitem{dubrovin1996hamiltonian}
{\sc Dubrovin, B., and Novikov, S.}
\newblock {Hamiltonian Formalism of One-Dimensional Systems of Hydrodynamic
  Type, and the Bogolyubov-Whitman Averaging Me}thod.
\newblock {\em Sov. Math. Dokl.\/} (1983), 665–--669.

\bibitem{dubrovin1998bihamiltonian}
{\sc Dubrovin, B., and Zhang, Y.}
\newblock Bihamiltonian hierarchies in 2d topological field theory at one-loop
  approximation.
\newblock {\em Comm. Math. Phys. 198}, 2 (1998), 311--361.

\bibitem{dubrovin1999frobenius}
{\sc Dubrovin, B., and Zhang, Y.}
\newblock {Frobenius manifolds and Virasoro constraints}.
\newblock {\em Selecta Math. 5}, 4 (1999), 423--466.

\bibitem{dubrovin2001normal}
{\sc Dubrovin, B., and Zhang, Y.}
\newblock Normal forms of hierarchies of integrable {PDE}s, {F}robenius
  manifolds and {Gromov-W}itten invariants.
\newblock {\em arXiv:math/0108160v1 [math.DG]\/} (2001).

\bibitem{dubrovin2004virasoro}
{\sc Dubrovin, B., and Zhang, Y.}
\newblock Virasoro symmetries of the extended {T}oda hierarchy.
\newblock {\em Comm. Math. Phys. 250}, 1 (2004), 161--193.

\bibitem{eguchi1997quantum}
{\sc Eguchi, T., Hori, K., and Xiong, C.-S.}
\newblock Quantum cohomology and {Virasoro} algebra.
\newblock {\em Phys. Lett. B 402}, 1-2 (1997), 71--80.

\bibitem{eguchi1995genus}
{\sc Eguchi, T., Yamada, Y., and Yang, S.-K.}
\newblock On the genus expansion in the topological string theory.
\newblock {\em Rev. Math. Phys. 7}, 03 (1995), 279--309.

\bibitem{fan2013witten}
{\sc Fan, H., Jarvis, T., and Ruan, Y.}
\newblock The {Witten} equation, mirror symmetry, and quantum singularity
  theory.
\newblock {\em Ann. Math.\/} (2013), 1--106.

\bibitem{ferapontov2001compatible}
{\sc Ferapontov, E.}
\newblock Compatible {Poi}sson brackets of hydrodynamic type.
\newblock {\em J. Phys. A 34}, 11 (2001), 2377.

\bibitem{getzler2001toda}
{\sc Getzler, E.}
\newblock The {To}da conjecture.
\newblock In {\em Symplectic Geometry And Mirror Symmetry}. World Scientific,
  2001, pp.~51--79.

\bibitem{givental2001gromov}
{\sc Givental, A.~B.}
\newblock {Gromov--Witten invariants and quantization of quadratic
  Hami}ltonians.
\newblock {\em Mosc. Math. J. 1}, 4 (2001), 551--568.

\bibitem{givental2001semisimple}
{\sc Givental, A.~B.}
\newblock Semisimple frobenius structures at higher genus.
\newblock {\em IMRN 2001}, 23 (2001), 1265--1286.

\bibitem{givental2005simple}
{\sc Givental, A.~B., and Milanov, T.~E.}
\newblock Simple singularities and integrable hierarchies.
\newblock In {\em The breadth of symplectic and Poisson geometry}. Springer,
  2005, pp.~173--201.

\bibitem{iglesias2021bi}
{\sc Iglesias, F.~H., and Shadrin, S.}
\newblock {Bi-Hamiltonian recursion, Liu-Pandharipande relations, and vanishing
  terms of the second Dubrovin-Zhang bra}cket.
\newblock {\em arXiv:2105.15138 [math-ph]\/} (2021).

\bibitem{kontsevich1992intersection}
{\sc Kontsevich, M.}
\newblock Intersection theory on the moduli space of curves and the matrix
  {Airy} function.
\newblock {\em Comm. Math. Phys. 147}, 1 (1992), 1--23.

\bibitem{liu2018lecture}
{\sc Liu, S.-Q.}
\newblock Lecture notes on bihamiltonian structures and their central
  invariants.
\newblock In {\em B-Model Gromov-Witten Theory}. Springer, 2018, pp.~573--625.

\bibitem{liu2015bcfg}
{\sc Liu, S.-Q., Ruan, Y., and Zhang, Y.}
\newblock {BCFG Drinfeld--Sokolov hierarchies and FJRW-t}heory.
\newblock {\em Invent. Math. 201}, 2 (2015), 711--772.

\bibitem{liu2020super}
{\sc Liu, S.-Q., Wang, Z., and Zhang, Y.}
\newblock Super tau-covers of bihamiltonian integrable hierarchies.
\newblock {\em arXiv:2009.01143 [math.DG]\/} (2020).

\bibitem{variational2}
{\sc Liu, S.-Q., Wang, Z., and Zhang, Y.}
\newblock Variational bihamiltonian cohomologies and integrable hierarchies II: Virasoro symmetries.
\newblock {arXiv:2109.01845 [math-ph]} 

\bibitem{liu2005deformations}
{\sc Liu, S.-Q., and Zhang, Y.}
\newblock Deformations of semisimple bihamiltonian structures of hydrodynamic
  type.
\newblock {J. Geom. Phys. 54}, 4 (2005), 427--453.

\bibitem{liu2011jacobi}
{\sc Liu, S.-Q., and Zhang, Y.}
\newblock Jacobi structures of evolutionary partial differential equations.
\newblock {\em Adv. Math. 227}, 1 (2011), 73--130.

\bibitem{liu2013bihamiltonian}
{\sc Liu, S.-Q., and Zhang, Y.}
\newblock Bihamiltonian cohomologies and integrable hierarchies {I}: a special
  case.
\newblock {\em Comm. Math. Phys. 324}, 3 (2013), 897--935.

\bibitem{lorenzoni2002deformations}
{\sc Lorenzoni, P.}
\newblock Deformations of b{i-Hamiltonian}structures of hydrodynamic type.
\newblock {\em J. Geom. Phys. 44}, 2-3 (2002), 331--375.

\bibitem{milanov2016gromov}
{\sc Milanov, T., Shen, Y., and Tseng, H.-H.}
\newblock Gromov--Witten theory of Fano orbifold curves, Gamma integral
  structures and ADE-Toda hierarchies.
\newblock {\em Geom. Topol. 20}, 4 (2016), 2135--2218.

\bibitem{witten1990structure}
{\sc Witten, E.}
\newblock On the structure of the topological phase of two-dimensional gravity.
\newblock {\em Nuclear Phys. B 340}, 2-3 (1990), 281--332.

\bibitem{witten1990two}
{\sc Witten, E.}
\newblock Two-dimensional gravity and intersection theory on moduli space.
\newblock {\em Surv. Differ. Geom. 1}, 1 (1990), 243--310.

\bibitem{zhang2002cp}
{\sc Zhang, Y.}
\newblock On the{ $CP^1$ topological sigma model and the Toda }lattice
  hierarchy.
\newblock {\em J. Geom. Phys. 40}, 3-4 (2002), 215--232.

\end{thebibliography}
\end{document}